\documentclass[final,onefignum,onetabnum]{siamart220329}

\usepackage{braket,amsfonts}

\usepackage{array}

\usepackage[caption=false]{subfig}
\usepackage{pgfplots}

\newsiamthm{claim}{Claim}
\newsiamremark{remark}{Remark}
\newsiamremark{hypothesis}{Hypothesis}
\crefname{hypothesis}{Hypothesis}{Hypotheses}

\usepackage{algorithmic}

\usepackage{graphicx,epstopdf}

\Crefname{ALC@unique}{Line}{Lines}

\usepackage{amsopn}

\usepackage{xspace}
\usepackage{bold-extra}
\usepackage[most]{tcolorbox}

\colorlet{texcscolor}{blue!50!black}
\colorlet{texemcolor}{red!70!black}
\colorlet{texpreamble}{red!70!black}
\colorlet{codebackground}{black!25!white!25}

\usepackage{lineno}
\usepackage{enumitem}
\usepackage{cleveref}

\DeclareMathOperator{\Tr}{Tr}
\newcommand{\m}{}
\newcommand{\simcom}{\textcolor{blue}}
\renewcommand{\vec}{\mathrm{vec}}

\newtheorem{thm}{Theorem}[section]
\newtheorem{lem}[thm]{Lemma}
\newtheorem{cor}[thm]{Corollary}
\newtheorem{prop}[thm]{Property}
\newtheorem{rem}[thm]{Remark}
\newtheorem{prob}[thm]{Problem}
\newtheorem{conj}[thm]{Conjecture}

\lstdefinestyle{siamlatex}{%
  style=tcblatex,
  texcsstyle=*\color{texcscolor},
  texcsstyle=[2]\color{texemcolor},
  keywordstyle=[2]\color{texemcolor},
  moretexcs={cref,Cref,maketitle,mathcal,text,headers,email,url},
}

\tcbset{%
  colframe=black!75!white!75,
  coltitle=white,
  colback=codebackground, 
  colbacklower=white, 
  fonttitle=\bfseries,
  arc=0pt,outer arc=0pt,
  top=1pt,bottom=1pt,left=1mm,right=1mm,middle=1mm,boxsep=1mm,
  leftrule=0.3mm,rightrule=0.3mm,toprule=0.3mm,bottomrule=0.3mm,
  listing options={style=siamlatex}
}

\newtcblisting[use counter=example]{example}[2][]{%
  title={Example~\thetcbcounter: #2},#1}

\newtcbinputlisting[use counter=example]{\examplefile}[3][]{%
  title={Example~\thetcbcounter: #2},listing file={#3},#1}

\DeclareTotalTCBox{\code}{ v O{} }
{ 
  fontupper=\ttfamily\color{black},
  nobeforeafter,
  tcbox raise base,
  colback=codebackground,colframe=white,
  top=0pt,bottom=0pt,left=0mm,right=0mm,
  leftrule=0pt,rightrule=0pt,toprule=0mm,bottomrule=0mm,
  boxsep=0.5mm,
  #2}{#1}

\patchcmd\newpage{\vfil}{}{}{}
\begin{tcbverbatimwrite}{tmp_\jobname_header.tex}
\title{Bounds on geodesic distances on the Stiefel manifold for a family of Riemannian metrics}

\author{Simon Mataigne\thanks{ICTEAM Institute, UCLouvain, Louvain-la-Neuve, Belgium. Simon Mataigne is a Research Fellow of the Fonds de la Recherche Scientifique - FNRS. This work was supported by the Fonds de la Recherche Scientifique - FNRS under Grant no T.0001.23.\\ \email{simon.mataigne@uclouvain.be}, \email{pa.absil@uclouvain.be}.}
\and P.-A. Absil\footnotemark[1]
\and Nina Miolane\thanks{UC Santa Barbara, Electrical and Computer Engineering, Santa Barbara, CA. N. M. acknowledges funding from the NSF Career 2240158. \email{ninamiolane@ucsb.edu}}}

\headers{Bounds on the geodesic distances on Stiefel}{Simon Mataigne, P.-A. Absil, Nina Miolane}
\end{tcbverbatimwrite}
\input{tmp_\jobname_header.tex}

\ifpdf
\hypersetup{ pdftitle={Bounds on the geodesic distances on Stiefel}}
\fi

\begin{document}
\maketitle


\begin{abstract}
We give bounds on geodesic distances on the Stiefel manifold, derived from new geometric insights. The considered geodesic distances are induced by the one-parameter family of Riemannian metrics introduced by H\"uper et al.~(2021), which contains the well-known Euclidean and canonical metrics. First, we give the best Lipschitz constants between the distances induced by any two members of the family of metrics.  Then, we give a lower and an upper bound on the geodesic distance by the easily computable Frobenius distance. We give explicit families of pairs of matrices that depend on the parameter of the metric and the dimensions of the manifold, where the lower and the upper bound are attained. These bounds aim at improving the theoretical guarantees and performance of minimal geodesic computation algorithms by reducing the initial velocity search space. 
In addition, these findings contribute to advancing the understanding of geodesic distances on the Stiefel manifold and their applications.
\end{abstract}



\section{Introduction} The past few years have seen the emergence of many applications for statistics on manifolds. The manifold-based statistical tools often rely on the notion of weighted mean, i.e., a minimizer of a weighted sum of the squared distances to the data points~\cite{frechet1948}. Optimization algorithms that compute weighted means on a Riemannian manifold $\mathcal{M}$ require an oracle that returns a \emph{minimal geodesic} between any two given points $\m{U},\widetilde{\m{U}}\in\mathcal{M}$, i.e., a shortest path between $\m{U}$ and $\widetilde{\m{U}}$~\cite{AfsariTronVidal2013}. While there are well-known closed-form solutions for minimal geodesics on some Riemannian manifolds such as the sphere, their computation remains a challenge in general, and in particular on the Stiefel manifold.  Given integers $n \geq p > 0$, the Stiefel manifold is the set of orthonormal $p$-frames in $\mathbb{R}^n$, written $\mathrm{St}(n,p)$, namely
\begin{equation}
    \mathrm{St}(n,p):=\{ \m{U}\in \mathbb{R}^{n\times p}\  |\ \m{U}^T\m{U} = \m{I}_p \}.
\end{equation}
This manifold is involved in several applications in statistics~\cite{ChakrabortyVemuri2019, Turaga08}, optimization~\cite{EdelmanArias98} and machine learning~\cite{Choromanski,GaoVaryAblinAbsil2022,Huang_Liu_Lang_Yu_Wang_Li_2018}. An implementation of the Stiefel manifold is available in software such as \href{https://geomstats.github.io/}{\texttt{Geomstats}}~\cite{JMLR:v21:19-027}, \href{https://www.manopt.org/}{\texttt{Manopt}}~\cite{manopt}, \href{https://manoptjl.org/stable/}{\texttt{Manopt.jl}}~\cite{Bergmann:2022}. 

On the Stiefel manifold endowed with the well-known Euclidean or canonical metrics~\cite{EdelmanArias98}, geodesics 
exist for all times~\cite{Huper2021}. It follows from the Hopf-Rinow theorem (see, e.g.,~\cite[Chap.~III,~Thm.~1.1]{sakai1996riemannian}) that at least one minimal geodesic exists between any two points. However, just computing \emph{some} geodesic, i.e., a \emph{locally} shortest path, between two given points on the Stiefel manifold is already a challenge, known as the \emph{geodesic endpoint problem}. Numerous numerical methods to address this problem have been proposed over recent years~\cite{BrynerDarshan17,mataigne2024efficient,Nguyengeodesics,RentmeestersQ,sutti2023shooting,Zimmermann17,ZimmermannRalf22}. Ensuring that the computed geodesic is minimal is even more challenging and is called the \emph{logarithm problem}. To date, no method is available that provably systematically achieves this goal and the aim of this work is to provide theoretical guarantees to these methods.
\paragraph{Contributions} After preliminaries in \cref{sec:preliminaries}, we show the bilipschitz equivalence between any two geodesic distances induced by a one-parameter family of Riemannian metrics in \cref{sec:equivalence}. The one-parameter family contains the Euclidean and canonical metrics. We derive, respectively in \cref{sec:general_lower_bound} and~\cref{sec:general_upper_bound}, a lower and an upper bound on these parameterized geodesic distances with the easily computable Frobenius distance
\begin{equation*}
    d_\mathrm{F}(U,\widetilde{U}) = \|U-\widetilde{U}\|_\mathrm{F}.
\end{equation*}
Determining if the bounds are attained or not depends on the pair $(n,p)$ and the parameter of the metric. In~\cref{sec:attaining_lower_bound} and~\cref{sec:reachingupperbound}, we discuss the influence of these parameters. Whenever we claim that the bound on the geodesic distance is attained, we prove it by providing families of pairs of matrices achieving this bound. In~\cref{sec:diameter}, we obtain the diameter of the Stiefel manifold under the Euclidean metric. In~\cref{sec:generalization_to_beta}, we deduce the bilipschitz equivalence of any geodesic distance from the parameterized family with the Frobenius distance. 

These bounds are useful in the design of minimal geodesic computation algorithms. Given initial and final points $U$ and $\widetilde{U}$ on the Stiefel manifold, and given a Riemannian metric and its induced geodesic distance function $d(\cdot,\cdot)$, knowing $m$ and $M$ such that $m\leq d(U,\widetilde{U})\leq M$ allows a  minimal geodesic computation algorithm to restrict the search for the initial velocity $\m{\Delta}$ to the shell $\mathcal{S}=\left\{\m{\Delta}\ | \ m\leq \|\m{\Delta}\|\leq M \right\}$. Without further information, finding a geodesic from $U$ to $\widetilde{U}$ with initial velocity in $\mathcal{S}$ does not guarantee that it is minimal. However, if the algorithm finds a curve from $U$ to $\widetilde{U}$ with length $m$, then the curve is a minimal geodesic. 
Moreover, if the algorithm finds a geodesic from $U$ to $\widetilde{U}$ with length greater than $M$, then the algorithm can issue the warning that the geodesic is not minimal, or pursue its quest for a (possibly) minimal geodesic.

\paragraph{Reproducibility statement} All the codes written to produce numerical results and figures are available at the address \url{https://github.com/smataigne/StiefelBounds.jl}.
\paragraph{Notation} For $p>0$, we denote the $p\times p$ identity matrix by $\m{I}_p$. For $n>p>0$, we define \begin{equation*}
    \m{I}_{n\times p }:=\begin{bmatrix}
        \m{I}_p\\
        \m{0}_{n-p\times p}
    \end{bmatrix}\text{ and } \m{I}_{p\times n} := \m{I}_{n\times p }^T.
\end{equation*}
$\mathrm{Skew}(p)$ is the set of $p\times p$ skew-symmetric matrices ($\m{A}=-\m{A}^T)$ and the orthogonal group of $p\times p$ matrices is denoted by
\begin{equation*}
    \mathrm{O}(p) :=\{\m{Q}\in\mathbb{R}^{p\times p}\ | \ \m{Q}^T\m{Q}=\m{I}_p\}.
\end{equation*}
The special orthogonal group $\mathrm{SO}(p)$ is the subset of matrices of $ \mathrm{O}(p)$ with positive unit determinant. An orthogonal completion of $U\in\mathbb{R}^{n\times p}$ is $U_\perp\in\mathbb{R}^{n\times n-p}$ such that $\begin{bmatrix}
    U&U_\perp
\end{bmatrix}\in \mathrm{O}(n)$. Throughout this paper, $\exp$ and $\log$ always denote the matrix exponential and the principal matrix logarithm. The Riemannian exponential and logarithm are written with capitals, $\text{Exp}$ and $\text{Log}$. The ambient space~$\mathbb{R}^{n\times p}$ of $\mathrm{St}(n,p)$ is equipped with the standard Frobenius metric, i.e., for $A,B\in \mathbb{R}^{n\times p}$, we have 
\begin{equation*}
	\langle A,B\rangle_\mathrm{F} := \Tr(A^T B).
\end{equation*}
Finally, we will handle submatrices and consider the standard notation from \cite[Chap.~1.3.3]{GoluVanl96}. $A_{i_1:i_2,j_1:j_2}$ denotes the submatrix obtained by selecting rows $i_1$ to $i_2$, columns $j_1$ to $j_2$ included.  $A_{i,j_1:j_2}$ means that only the $i$th row is selected and $A_{j_1:j_2}$ that all rows are selected.

\section{Preliminaries on the Stiefel manifold}\label{sec:preliminaries}
This section is based on~\cite{AbsMahSep2008,EdelmanArias98,ZimmermannRalf22}. The dimension of $\mathrm{St}(n,p)$ is $np-\frac{p(p+1)}{2}$~\cite{AbsMahSep2008}. It is a differentiable manifold and its tangent space at a point $U\in\mathrm{St}(n,p)$ can be written as
\begin{align*}
    T_U \mathrm{St}(n,p) &=\{\m{\Delta}\in\mathbb{R}^{n\times p}\ | \ U^T\m{\Delta}+\m{\Delta}^TU=\m{0} \}\\
    &= \{ \m{UA}  + U_\perp \m{B} \in \mathbb{R}^{n\times p} \ | \ \m{A}\in \mathrm{Skew}(p), \ \m{B} \in \mathbb{R}^{(n-p)\times p}\}.
\end{align*}
Notice that for a given $\m{\Delta}\in T_U \mathrm{St}(n,p)$, $\m{A}$ is uniquely defined while $\m{B}$ is not since, for all $\m{R}\in\mathrm{O}(n-p)$, $U_\perp \m{B}=(U_\perp \m{R})(\m{R}^T \m{B})$ and $U_\perp R$ remains an orthogonal completion of $U$.
\subsection{Metrics and distances} On a manifold, it is customary for the distance to be induced by a Riemannian metric (or \emph{metric} for short). The Stiefel manifold admits two classical metrics~\cite{EdelmanArias98}, briefly reviewed next.
The first is the \emph{Euclidean metric}. This metric is inherited from the ambient Euclidean space $\mathbb{R}^{n\times p}$ and given by
\begin{equation*}
    \langle \m{\Delta},\widetilde{\m{\Delta}}\rangle_{\mathrm{E}} = \Tr ( \m{\Delta}^T\widetilde{\m{\Delta}})
    = \Tr (\m{A}^T\widetilde{\m{A}}) +\Tr (\m{B}^T\widetilde{\m{B}}) ,
\end{equation*}
for all  $\m{\Delta} =  \m{UA} +U_\perp \m{B}$ and $\widetilde{\m{\Delta}} = U \widetilde{\m{A}}+ U_\perp \widetilde{\m{B}}$ in $ T_U \mathrm{St}(n,p)$.
The second is the \emph{canonical metric}. It is inherited from the quotient geometry $\mathrm{St}(n,p) = \mathrm{O}(n)/\mathrm{O}(n-p)$~\cite{EdelmanArias98} and it is given by
\begin{equation*}
     \langle \m{\Delta},\widetilde{\m{\Delta}}\rangle_{\mathrm{c}} = \Tr (\m{\Delta}^T(\m{I}_n-\frac{1}{2}\m{UU}^T)\widetilde{\m{\Delta}})= \frac{1}{2}\Tr (\m{A}^T\widetilde{\m{A}} )+\Tr (\m{B}^T\widetilde{\m{B}}). 
\end{equation*}
These two metrics are subsumed by the family of metrics introduced in~\cite{Huper2021}. For convenience, we use another parameterization of the family than the authors and define the $\beta$-metric, with $\beta>0$, by 
\begin{equation}\label{eq:beta_metric}
    \langle\m{\Delta},\widetilde{\m{\Delta}}\rangle_\beta = \Tr (\m{\Delta}^T(\m{I}_n-(1-\beta)\m{UU}^T)\widetilde{\m{\Delta}})=\beta \Tr( \m{A}^T\widetilde{\m{A}}) +\Tr( \m{B}^T\widetilde{\m{B}}).
\end{equation}
This $\beta$-parameterization has already been used~\cite{absil2024ultimate,mataigne2023bounds, mataigne2024efficient, Nguyengeodesics}. The Euclidean and canonical metrics correspond to $\beta = 1$ and $\beta = \frac{1}{2}$ respectively. The Riemannian submersion theory offers a geometric interpretation of this metric for all $\beta>0$~\cite{absil2024ultimate, Huper2021}. The norm induced by the $\beta$-metric is $\|\m{\Delta} \|_\beta= \sqrt{\langle \m{\Delta},\m{\Delta}\rangle_\beta}$, and the length of a continuously differentiable curve $\gamma:[0,1]\mapsto\mathrm{St}(n,p)$ is \begin{equation}\label{eq:length}
    l_\beta(\gamma) = \int_0^1 \|\Dot{\gamma}(t)\|_\beta \ \mathrm{d}t,
\end{equation}
where $\Dot{\gamma}$ denotes the time derivative of $\gamma$. Finally, for all $U,\widetilde{U}\in \mathrm{St}(n,p)$, the distance (also termed \emph{geodesic distance}) is 
\begin{equation}\label{eq:geodesicdistance}
     d_\beta(U,\widetilde{U})=\inf_\gamma\{l_\beta(\gamma)\ |\  \gamma(0)=U,\ \gamma(1) = \widetilde{U}\}.
 \end{equation}
  Since $\beta$-geodesics exist for all times~\cite{Huper2021}, the ``$\inf$'' in \eqref{eq:geodesicdistance} is in fact a ``$\min$'', by the Hopf-Rinow theorem.
 \begin{figure}
     \centering
     \includegraphics[width = 7cm]{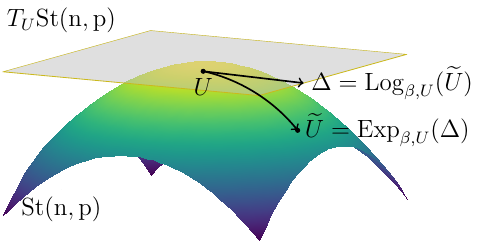}
     \caption{An artist view of the Stiefel manifold $\mathrm{St}(n,p)$, the tangent space $T_U\mathrm{St}(n,p)$, the Riemannian exponential and logarithm.}
     \label{fig:exp_log}
 \end{figure}
 
 \subsection{The Riemannian exponential and logarithm} 
 Exponential and logarithmic maps are the hammer and anvil on a manifold. The first allows to travel in a ``straight line'', i.e., along a geodesic. The second indicates which geodesic to follow to connect two points via the shortest path. 
 
 Formally, the Riemmanian exponential is the function $\mathrm{Exp}_{\beta,U}:T_U\mathrm{St}(n,p)\mapsto\mathrm{St}(n,p)$  mapping $\m{\Delta}$ to the point reached at unit time by the geodesic with starting point $U$ and initial velocity $\m{\Delta}$~(see, e.g., \cite[Chap.~II, Sec.~2]{sakai1996riemannian}). The Riemannian logarithm $\mathrm{Log}_{\beta,U}(\widetilde{U})$ is the function that returns the set of all minimal-norm tangent vectors $\Delta\in T_U\mathrm{St}(n,p)$ such that $\mathrm{Exp}_{\beta,U}(\m{\Delta})=\widetilde{U}$. These concepts are illustrated in~\cref{fig:exp_log}. 
An explicit expression for the Riemannian exponential is given next.
  \begin{thm}{The Riemannian exponential~\cite[Eq.~10]{ZimmermannRalf22}.}\label{thm:geodesics}
    \it
    For all $U\in \mathrm{St}(n,p)$ and $\m{\Delta} \in T_U\mathrm{St}(n,p)$, 
    \begin{equation}\label{eq:paramgeogeneral}
        \mathrm{Exp}_{\beta,U}(\m{\Delta}) = \begin{bmatrix}
            U & \ \m{Q}
        \end{bmatrix}\exp\left(\begin{bmatrix}
            2\beta \m{A}& -\m{B}^T\\
            \m{B} & \m{0}
        \end{bmatrix}\right)\m{I}_{n\times p}\exp\left((1-2\beta)\m{A}\right),
    \end{equation}
    where $\m{A} := U^T \m{\Delta}\in \mathrm{Skew}(p)$ and $\m{Q}\m{B} := (\m{I}-UU^T)\m{\Delta}\in\mathbb{R}^{n\times p}$ is any matrix decomposition where $\m{Q}\in\mathrm{St}(n,n-p)$, $\m{Q}^T U=\m{0}$ and $\m{B}\in\mathbb{R}^{n-p\times p}$.
\end{thm}
If $n>2p$, we have $\mathrm{rank}(B)\leq \min(n-p, p)=p$. This yields a reduced decomposition $QB = \widetilde{Q}\widetilde{B}$ where $\widetilde{Q}\in\mathrm{St}(n,p)$ and $\widetilde{B}\in\mathbb{R}^{p\times p}$. Therefore,~\cref{thm:geodesics} admits a reduced formulation where $\m{Q}\in\mathrm{St}(n,p)$ and $\m{B}\in\mathbb{R}^{p\times p}$~\cite{RentmeestersQ, ZimmermannRalf22}.
The first matrix exponential of \eqref{eq:paramgeogeneral} belongs then to $\mathbb{R}^{2p\times 2p}$ instead of $\mathbb{R}^{n\times n}$. This observation brings significant computational speed-up for large scale applications.

\begin{prob}{The Logarithm problem~\cite[Prob.~2.2]{mataigne2024efficient}.}\label{prob:truelogproblem}
    \it
    Let $U,\m{\widetilde{U}}\in \mathrm{St}(n,p)$ and $\beta>0$. The Riemannian logarithm $\mathrm{Log}_{\beta,U}(\widetilde{U})$ returns all $\m{\Delta}\in T_{U}\mathrm{St}(n,p)$ such that
    \begin{equation*}
         \mathrm{Exp}_{\beta,U}(\m{\Delta})=\widetilde{U} \text{ and } \|\m{\Delta}\|_\beta = d_\beta(U,\widetilde{U}).
    \end{equation*}
    The curves $[0,1]\ni t\mapsto \mathrm{Exp}_{\beta,U}(t\m{\Delta})$ are then \emph{minimal geodesics}.
\end{prob}
On compact Riemannian manifolds, every geodesic has a \emph{cut time} $t^*$~\cite{sakai1996riemannian}: the geodesic stays minimal as long as $t \leq t^*$, not beyond. The point $\mathrm{Exp}_{\beta,U}(t^*\m{\Delta})$ is termed the \emph{cut point}. For a given point $U$, the set of all cut points is called the \emph{cut locus}. A shooting direction $\Delta$ is thus a Riemannian logarithm if and only if the destination is reached along the geodesic before or when visiting the cut point. On the Stiefel manifold, the cut point is known only for some specific geodesics. A simpler problem is to study the distance to the \emph{nearest} cut point, called the injectivity radius $\mathrm{inj}_{\mathrm{St}(n,p),\beta}$. Strictly within this radius,~\cref{prob:truelogproblem} is guaranteed to have a unique solution. Recent works~\cite{absil2024ultimate, stoye2024injectivity, zimmermann2024high} showed that $0.894\pi<\mathrm{inj}_{\mathrm{St}(n,p),\beta}<0.914\pi$ when $\beta = \frac{1}{2}$, $\mathrm{inj}_{\mathrm{St}(n,p),\beta}=\pi$ when $\beta=1$~\cite{zimmermann2024injectivity} and $\mathrm{inj}_{\mathrm{St}(n,p),\beta}\leq \pi$ for all $\beta>0$. However, as shown in~\cref{sec:diameter}, the diameter of Stiefel grows with $p$ and this jeopardizes the relevance of the injectivity radius. There is thus a need to provide new insights, allowing to assert the quality of candidates for the Riemannian logarithm.
 
 \section{Equivalence of the geodesic $\beta$-distances}\label{sec:equivalence} An important step to study the geodesic distances induced by the $\beta$-metrics is to understand how these distances relate to each other as $\beta$ varies. In this section, we prove that for all pairs $\beta_1,\beta_2>0$, the distances $d_{\beta_1}$ and $d_{\beta_2}$ are bilipschitz equivalent, i.e., there is $\alpha,\delta>0$ such that for all $U,\widetilde{U}\in \mathrm{St}(n,p)$, we have $\alpha d_{\beta_1}(U,\widetilde{U})\leq d_{\beta_2}(U,\widetilde{U})\leq \delta d_{\beta_1}(U,\widetilde{U})$. 

We start with an illustrative numerical experiment to motivate our goal. We use~\cite[Alg.~1~and~4]{ZimmermannRalf22} to estimate the Euclidean ($\beta_1=1$) and canonical ($\beta_2=\frac{1}{2}$) geodesic distances between randomly generated matrices of $\mathrm{St}(4,2)$. The results are shown in~\cref{fig:equivalence}. We can guess the best Lipschitz constants in this case, $\alpha=\frac{\sqrt{2}}{2}$ and $\delta = 1$ shown by a solid black line and a dashed green line in~\cref{fig:equivalence} respectively. 

For small dimensions, such as $(n,p)=(4,2)$, numerical estimates may very well be exact (up to machine precision), which is why~\cref{fig:equivalence} and further ones tend to obey the theoretical results that we introduce in this paper. However, this does not hold true for large dimensions that one encounters in applications where the algorithms can significantly overestimate the distance by failing to compute a geodesic that is minimal. In fact, even for small dimension, we will observe in~\cref{fig:canonicalsamples} failures of the algorithms. 

\cref{lem:firstlem} is the key observation which leads to the bilipschitz equivalence. It holds for any pair of Riemannian metrics on the Stiefel manifold, which is why we introduce subscripts ``$\mathrm{a},\mathrm{b}$'' instead of ``$\beta_1,\beta_2$''.
\begin{figure}
    \centering
    \includegraphics[width = 8cm]{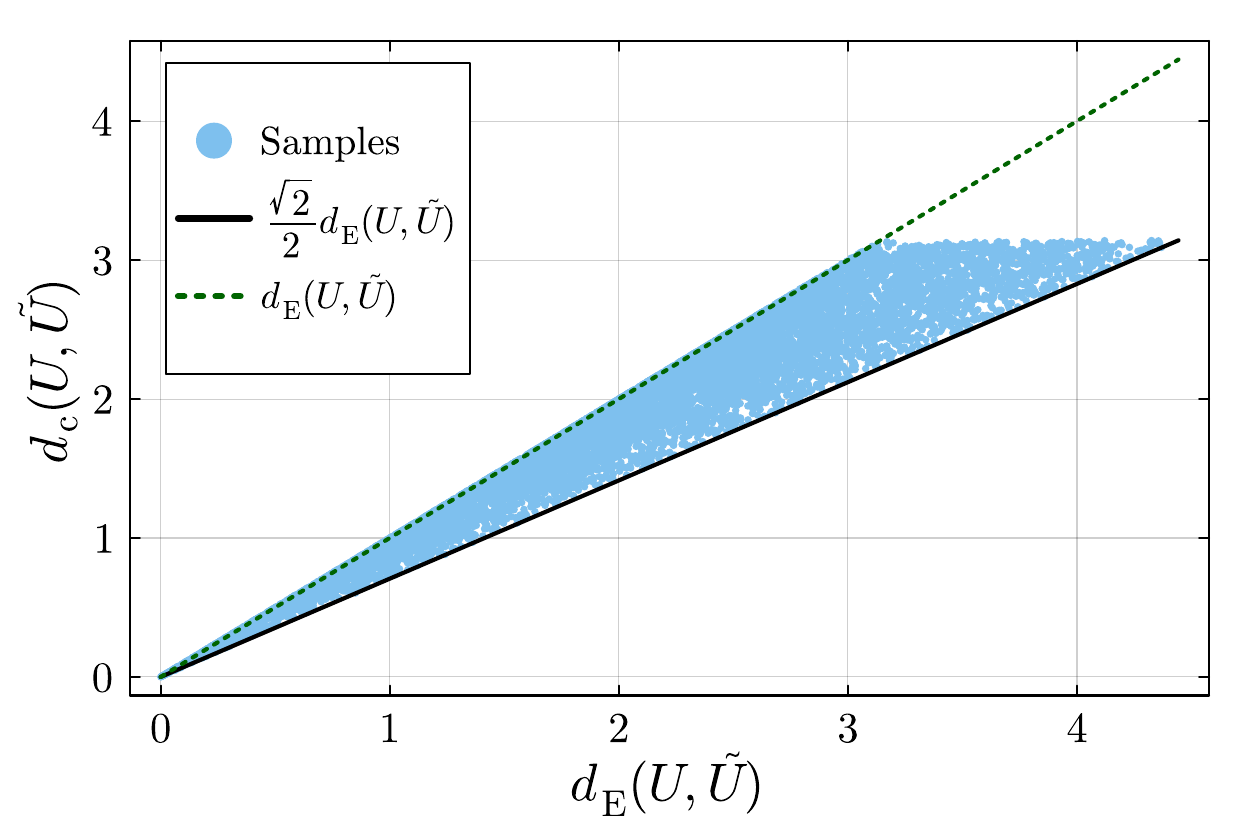}
    \caption{Bounds between the canonical and the Euclidean distance on the Stiefel manifold $\mathrm{St}(4,2)$: $\frac{\sqrt{2}}{2} d_{\mathrm{E}}(U,\widetilde{U})\leq d_{\mathrm{c}}(U,\widetilde{U})\leq d_{\mathrm{E}}(U,\widetilde{U})$ for all $U,\widetilde{U}\in \mathrm{St}(4,2)$ (see~\cref{cor:special}). The plot shows the canonical distances $d_\mathrm{c}$ ($\beta = \frac{1}{2}$) and the Euclidean distances $d_\mathrm{E}$ ($\beta=1$) estimated for 10000 randomly generated pairs $(U,\widetilde{U})$ on $\mathrm{St}(4,2)$.}
    \label{fig:equivalence}
\end{figure}

\begin{lem}\label{lem:firstlem}
    \it
    Let $\langle\cdot,\cdot\rangle_{\mathrm{a}}$ and $\langle\cdot,\cdot\rangle_{\mathrm{b}}$ be two Riemannian metrics on $\mathrm{St}(n,p)$. If $\|\m{\Delta}\|_\mathrm{a}\leq \mu \|\m{\Delta}\|_\mathrm{b}$ for all $U\in\mathrm{St}(n,p)$ and $\m{\Delta}\in T_U\mathrm{St}(n,p)$, then $d_\mathrm{a}(U,\widetilde{U})\leq \mu d_\mathrm{b}(U,\widetilde{U})$ for all $U,\widetilde{U}\in\mathrm{St}(n,p)$.
\end{lem}

\begin{proof}
    Let $U,\widetilde{U}\in\mathrm{St}(n,p)$. For every differentiable curve $\gamma:[0,1]\rightarrow \mathrm{St}(n,p)$ with $\gamma(0)=U$ and $\gamma(1)=\widetilde{U}$, it holds that
    \begin{equation*}
        d_\mathrm{a}(U,\widetilde{U})\leq l_\mathrm{a}(\gamma)=\int_0^1\|\Dot{\gamma}(t)\|_\mathrm{a}\ \mathrm{d}t\leq \mu \int_0^1\|\Dot{\gamma}(t)\|_\mathrm{b}\ \mathrm{d}t=\mu l_\mathrm{b}(\gamma).
    \end{equation*}
    Hence, by definition~\eqref{eq:geodesicdistance}, we obtain
    \begin{align*}
        d_\mathrm{a}(U,\widetilde{U})\leq \mu \inf\{l_b(\gamma)\ | \gamma(0)=U,\ \gamma(1)=\widetilde{U}\}=\mu d_\mathrm{b}(U,\widetilde{U}).
    \end{align*} 
\end{proof}

 We show in~\cref{thm:equivalence} that the bilipschitz equivalence of the $\beta$-distances follows from~\cref{lem:firstlem} since the $\beta$-metrics satisfy its hypotheses.
\begin{thm}\label{thm:equivalence}
    \it
     For all $U,\widetilde{U}\in\mathrm{St}(n,p)$,
    \begin{enumerate}
        \item[(a)] If $0<\beta_2\leq\beta_1$, then $d_{\beta_1}(U,\widetilde{U})\leq \sqrt{\frac{\beta_1}{\beta_2}}d_{\beta_2}(U,\widetilde{U})$.
        \item[(b)] If $0<\beta_1\leq\beta_2$, then $d_{\beta_1}(U,\widetilde{U})\leq d_{\beta_2}(U,\widetilde{U})$.
    \end{enumerate}
\end{thm}

\begin{proof}
        (a) Let $0<\beta_2 \leq \beta_1$. For all $\widehat{U}\in \mathrm{St}(n,p)$ and all $\m{\Delta}\in T_{\widehat{U}}\mathrm{St}(n,p)$, we have
        \begin{equation}
            \|\m{\Delta}\|_{\beta_1}^2 = \beta_1\|\m{A}\|_{\mathrm{F}}^2+\|\m{B}\|_{\mathrm{F}}^2\leq \beta_1 \|\m{A}\|_{\mathrm{F}}^2+\frac{\beta_1}{\beta_2}\|\m{B}\|_{\mathrm{F}}^2=\frac{\beta_1}{\beta_2}\|\m{\Delta}\|_{\beta_2}^2,
        \end{equation} hence $ \|\m{\Delta}\|_{\beta_1}^2 \leq \frac{\beta_1}{\beta_2}\|\m{\Delta}\|_{\beta_2}^2$. By~\cref{lem:firstlem}, the claim follows.
        
        (b) Let $0<\beta_1 \leq \beta_2$. For all $\widehat{U}\in \mathrm{St}(n,p)$ and all $ \m{\Delta} \in T_{\widehat{U}}\mathrm{St}(n,p)$, we have
        \begin{equation}
            \|\m{\Delta}\|_{\beta_1}^2 = \beta_1\|\m{A}\|_{\mathrm{F}}^2+\|\m{B}\|_{\mathrm{F}}^2\leq \beta_2 \|\m{A}\|_{\mathrm{F}}^2+\|\m{B}\|_{\mathrm{F}}^2=\|\m{\Delta}\|_{\beta_2}^2,
        \end{equation} hence $ \|\m{\Delta}\|_{\beta_1}^2 \leq \|\m{\Delta}\|_{\beta_2}^2$. By~\cref{lem:firstlem}, the claim follows.
\end{proof}

\cref{cor:allbounds} summarizes the bilipchitz equivalence between all the geodesic distances under the $\beta$-metrics.

\begin{cor}\label{cor:allbounds}
    \it
    $\min\left\{1,\sqrt{\frac{\beta_1}{\beta_2}}\right\} d_{\beta_2}(U,\widetilde{U})\leq d_{\beta_1}(U,\widetilde{U})\leq \max\left\{1, \sqrt{\frac{\beta_1}{\beta_2}}\right\} d_{\beta_2}(U,\widetilde{U})$\linebreak for all $U,\widetilde{U}\in \mathrm{St}(n,p).$
\end{cor}

\begin{proof}
Combine~\cref{thm:equivalence}.(a) and~\cref{thm:equivalence}.(b).
\end{proof}

As a special case, we write~\cref{cor:allbounds} for the popular choice of the Euclidean ($\beta=1$) and the canonical ($\beta =\frac{1}{2}$) geodesic distances. This leads us to recover the Lipschitz constants that we conjectured in~\cref{fig:equivalence}.
\begin{cor}\label{cor:special}
    \it
    $\frac{\sqrt{2}}{2} d_{\mathrm{E}}(U,\widetilde{U})\leq d_{\mathrm{c}}(U,\widetilde{U})\leq d_{\mathrm{E}}(U,\widetilde{U})\leq \sqrt{2}d_{\mathrm{c}}(U,\widetilde{U})$\linebreak for all $U,\widetilde{U}\in \mathrm{St}(n,p)$.
\end{cor}

\begin{proof}
    Recall that $d_{\mathrm{E}}=d_{\beta=1}$ and $d_{\mathrm{c}}=d_{\beta=\frac{1}{2}}$.
\end{proof}
We show that the Lipschitz constants of~\cref{thm:equivalence} are tight, as hinted by the results of our numerical experiment in~\cref{fig:equivalence}. This means that in~\cref{cor:allbounds}, $\min\left\{1,\sqrt{\frac{\beta_1}{\beta_2}}\right\}$ is the maximum coefficient for a lower bound and $\max\left\{1,\sqrt{\frac{\beta_1}{\beta_2}}\right\}$ is the minimum coefficient for an upper bound, both of which hold true for all pairs $U,\widetilde{U}\in\mathrm{St}(n,p)$.

\begin{thm}\label{thm:tightness}
    \it
    There is respectively $U,\widetilde{U}\in\mathrm{St}(n,p)$ such that\\
    \begin{enumerate}
        \item[(a)]  the bound of~\cref{thm:equivalence}.(a)
        is attained;  
        \item[(b)]  the bound of~\cref{thm:equivalence}.(b)
        is attained.    
    \end{enumerate}
\end{thm}
\begin{proof}
        (a) Let $0<\beta_2\leq\beta_1$, $U\in \mathrm{St}(n,p)$ and $\m{\Delta}=\m{UA}+U_\perp \m{B} \in T_U\mathrm{St}(n,p)$. Take $\m{B} =0$ and observe that~\eqref{eq:paramgeogeneral} reduces to 
        \begin{align*}
            \mathrm{Exp}_{\beta,U} (t\m{\Delta})
            &= U e^{\m{A}t},
        \end{align*}
        where we have used the definition of the matrix exponential and the fact that $U^T U = \m{I}_p$. Hence $ \gamma(t):=\mathrm{Exp}_{\beta,U} (t\Delta)$ does not depend on $\beta$. Further take $\m{A}$ small enough for $\gamma$ to be minimal between $U =  \gamma(0)$ and $\widetilde{U}:= \gamma(1)$. Then \begin{equation*}
            d_{\beta_1}(U,\widetilde{U}) = l_{\beta_1}(\gamma) = \|\m{\Delta}\|_{\beta_1} = \sqrt{\frac{\beta_1}{\beta_2}}\|\m{\Delta}\|_{\beta_2} = \sqrt{\frac{\beta_1}{\beta_2}} l_{\beta_2}(\gamma)=\sqrt{\frac{\beta_1}{\beta_2}}d_{\beta_2}(U,\widetilde{U}).
        \end{equation*} 
        (b) Let $0<\beta_1\leq\beta_2$, $U\in \mathrm{St}(n,p)$ and $\m{\Delta}=\m{UA}+U_\perp \m{B} \in T_U\mathrm{St}(n,p)$. Take $\m{A}=0$ and observe that~\eqref{eq:paramgeogeneral} reduces to \begin{align*}
             \mathrm{Exp}_{\beta,U} (t\m{\Delta}) 
            &= \begin{bmatrix}
                U&\ U_\perp 
            \end{bmatrix}
            \exp\left(t\begin{bmatrix}
                \m{0}&-\m{B}^T\\
                \m{B}&\m{0}
            \end{bmatrix}\right)\m{I}_{n\times p}.
        \end{align*}
        Again, $\gamma(t):= \mathrm{Exp}_{\beta,U} (t\Delta)$ does not depend on $\beta$. Further take $\m{B}$ small enough for $\gamma$ to be minimal between $U = \gamma(0)$ and $\widetilde{U}:=\gamma(1)$. Then 	    \begin{equation}
        \nonumber
            d_{\beta_1}(U,\widetilde{U}) = l_{\beta_1}(\gamma_\beta) = \|\m{\Delta}\|_{\beta_1} = \|\m{\Delta}\|_{\beta_2} =  l_{\beta_2}(\gamma_\beta)=d_{\beta_2}(U,\widetilde{U}).
        \end{equation}
        This concludes the proof.
\end{proof}

The proof of~\cref{thm:tightness} shows that the bounds are attained in a neighborhood of the origin in the $(d_{\beta_1},d_{\beta_2})$-plane. In the absence of a lower bound on the cut time of the geodesic $\gamma$ considered in the proof, the size of the neighborhood is unknown.
\section{Bounds on the {$\beta$}-distance by the Frobenius distance: overview}\label{sec:introbounds} 
We have shown that all the $\beta$-distances are bilipschitz equivalent. Nonetheless, it remains that computing any such distance is a computational challenge~\cite{mataigne2024efficient, ZimmermannRalf22}. Therefore, we wish to characterize the $\beta$-distance in terms of a much easier quantity: the Frobenius distance. Again, we start by showing the results of a numerical experiment. We generate random samples on $\mathrm{St}(n,p)$ and we estimate both the Euclidean and the canonical geodesic distance, respectively $d_\mathrm{E}$ ($\beta = 1$) and $d_\mathrm{c}$ ($\beta = \frac{1}{2}$) using~\cite[Alg.~1~and~4]{ZimmermannRalf22}. 
The estimates are shown by the dots in~\cref{fig:shootingfig}. The latter suggest a structured relation between the geodesic distance and the Frobenius distance. In particular, for all $\delta\in[0,2\sqrt{p}]$, we wish to describe analytically
\begin{align}
\label{eq:true_lower_bound}
m_\beta(\delta)&:=\min_{U,\widetilde{U}\in\mathrm{St}(n,p)}\{d_\beta(U,\widetilde{U})\ | \ \|\widetilde{U}-U\|_\mathrm{F} =  \delta\}\\
\label{eq:true_upper_bound}
\text{ and }  M_\beta(\delta)&:=\max_{U,\widetilde{U}\in\mathrm{St}(n,p)}\{d_\beta(U,\widetilde{U})\ | \ \|\widetilde{U}-U\|_\mathrm{F} =  \delta\}.
\end{align} 
We justify in~\cref{lem:frobenius_diameter} that it is enough to consider $\delta\in[0,2\sqrt{p}]$.
\begin{lem}\label{lem:frobenius_diameter}
	Let $n>p$. For all $U,\widetilde{U}\in\mathrm{St}(n,p)$, we have $\|U-\widetilde{U}\|_\mathrm{F}\in[0, 2\sqrt{p}]$ and for all $\delta\in[0,2\sqrt{p}]$, there is $U,\widetilde{U}$ such that $\delta=\|U-\widetilde{U}\|_\mathrm{F}$.
\end{lem}
\begin{proof}
Since $\|U\|_\mathrm{F}=\|\widetilde{U}\|_\mathrm{F}=\sqrt{p}$, we have $\|U-\widetilde{U}\|_\mathrm{F}\leq  \|U\|_\mathrm{F}+\|\widetilde{U}\|_\mathrm{F}=2\sqrt{p}$. The inequality is attained for $\widetilde{U}=-U$. Moreover, $\mathrm{St}(n,p)$ is connected if $n>p$. Hence, $U$ and $-U$ are connected by a curve $\gamma:[0,1]\rightarrow\mathrm{St}(n,p)$. By the intermediate value theorem, $[0,1]\ni t \mapsto \|U-\gamma(t)\|_\mathrm{F}$ achieves every value in $[0,2\sqrt{p}]$.
\end{proof}
In~\cref{sec:lowerbound}, we obtain a lower bound $\widehat{m}_\beta$ such that $\widehat{m}_\beta(\delta)\leq m_\beta(\delta)$ for all $\delta\in[0,2\sqrt{p}]$. In~\cref{sec:attaining_lower_bound}, we discuss when $\widehat{m}_\beta(\delta)= m_\beta(\delta)$ for all $\delta\in[0,2\sqrt{p}]$, $\widehat{m}_\beta = m_\beta$ in short, and we give pairs of matrices $U,\widetilde{U}$ achieving~\eqref{eq:true_lower_bound}. For the upper bound $M_\beta$, we start by considering $\beta=1$ (Euclidean metric) in~\cref{sec:general_upper_bound}. We obtain an analytic expression for $\widehat{M}_\mathrm{E}$ such that $\widehat{M}_\mathrm{E}(\delta)\geq M_\mathrm{E}(\delta)$ for all $\delta\in[0,2\sqrt{p}]$. Similarly, we discuss in~\cref{sec:reachingupperbound} when $\widehat{M}_\mathrm{E} = M_\mathrm{E}$ and we give the pairs of matrices $U,\widetilde{U}$ achieving~\eqref{eq:true_upper_bound}. In~\cref{sec:generalization_to_beta}, we leverage the bilipschitz equivalence of the $\beta$-distances to obtain $\widehat{M}_\beta$ for all $\beta>0$.
\begin{figure}
    \centering
    \includegraphics[width = 6.4cm]{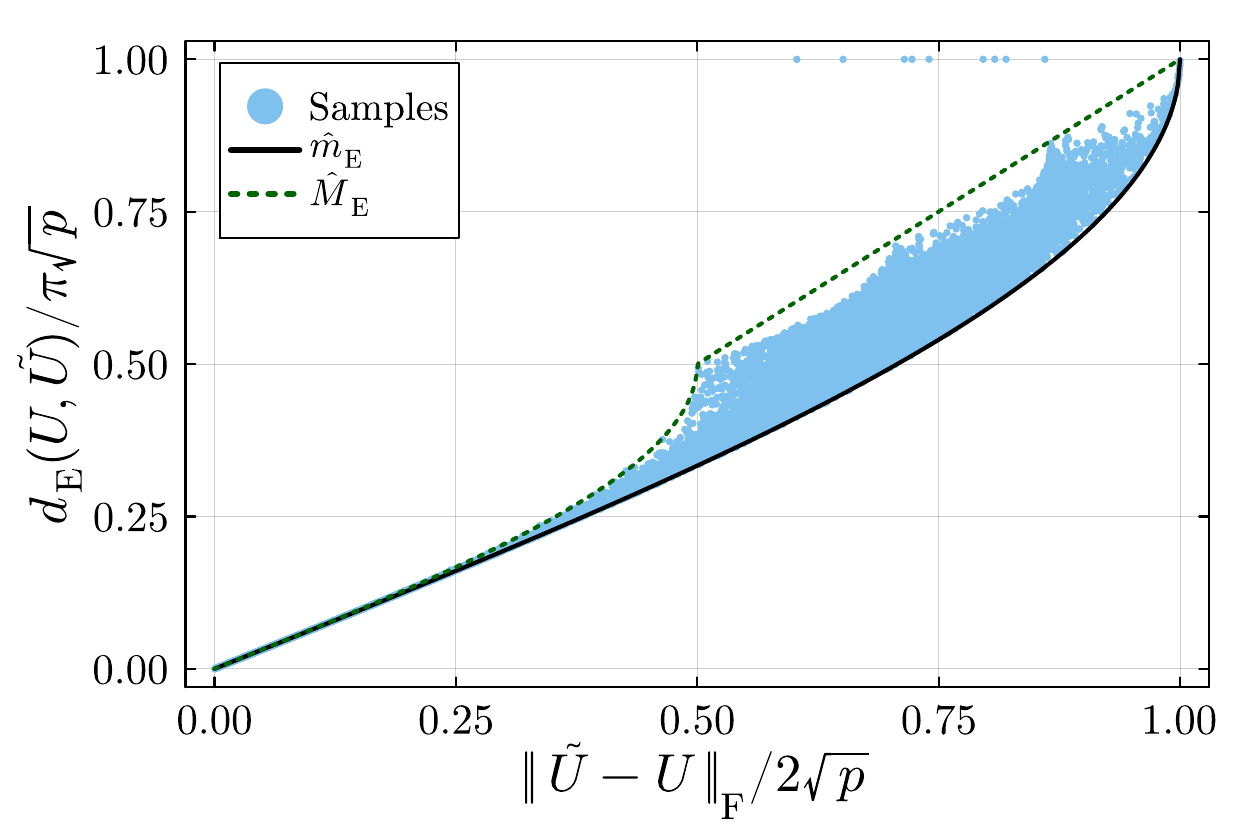}
    \includegraphics[width = 6.4cm]{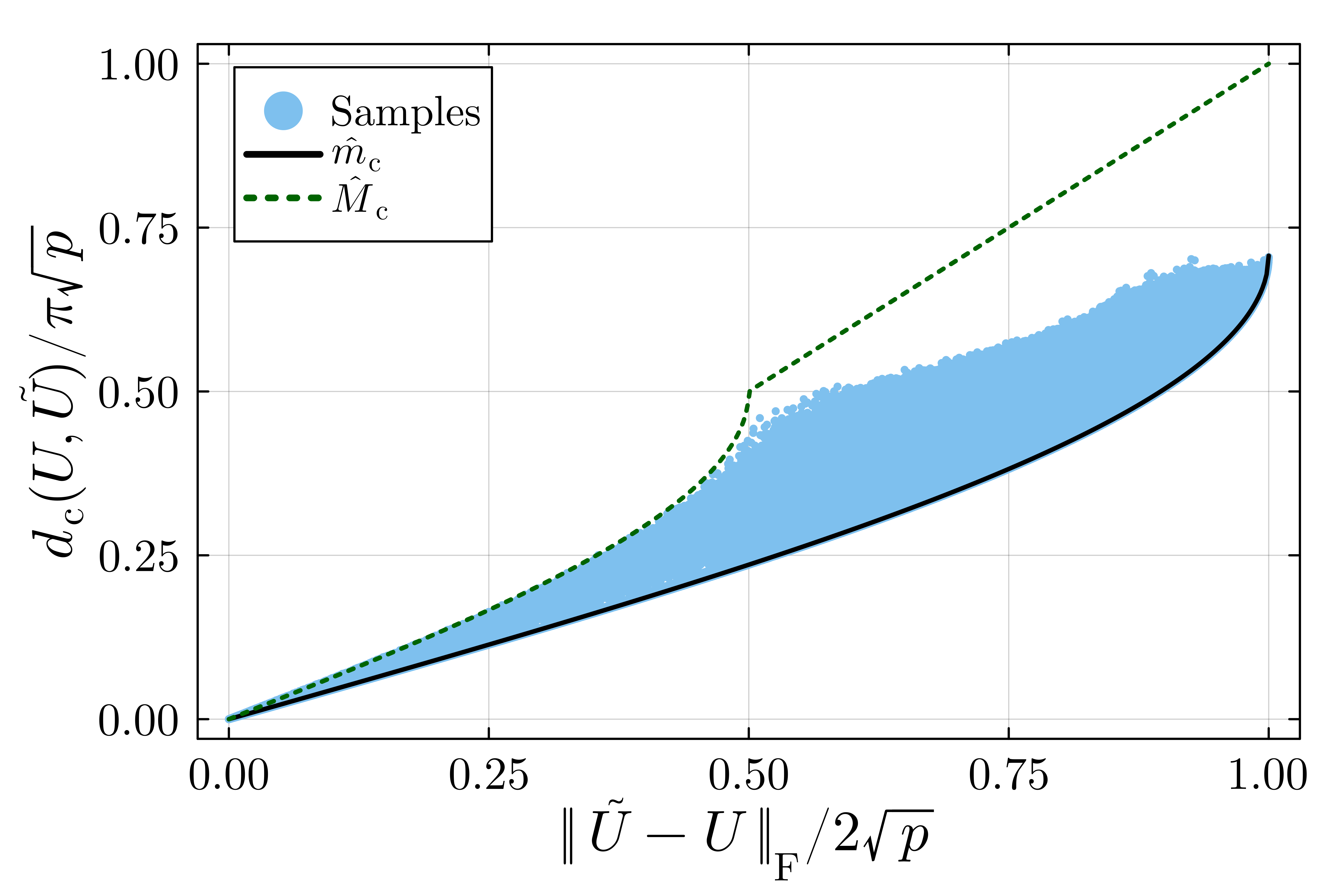}
    \caption{Bounds between the Frobenius distance and, on the left, the Euclidean ($\beta=1$) geodesic distance on $\mathrm{St}(5,4)$ and on the right, the canonical ($\beta=\frac{1}{2}$) geodesic distance on $\mathrm{St}(8,4)$: $\widehat{m}_\beta(\delta)\leq d_{\beta}(U,\widetilde{U})\leq \widehat{M}_\beta(\delta)$ for all $U,\widetilde{U}\in \mathrm{St}(n,p)$ with $\|\widetilde{U}-U\|_\mathrm{F} =  \delta$ (see \cref{thm:generalizedbounds}). The plots show the Frobenius distances and the geodesic distances $d_\mathrm{\beta}$ estimated for 20 000, respectively 1 000 000, randomly generated pairs $(U,\widetilde{U})$ on $\mathrm{St}(5,4)$, respectively $\mathrm{St}(8,4)$.}
    \label{fig:shootingfig}
\end{figure}

\section{A lower bound on the $\beta$-distance by the Frobenius distance}\label{sec:general_lower_bound}
Given $U\in \mathrm{St}(n,p)$, define $\vec(U)\in\mathbb{R}^{np}$ as the vector obtained by stacking the columns of $U$ one under the other such that $U_{i,j}=\vec(U)_{(j-1)n+i}$ for $i,j\geq 1$. Let $\vec(\mathrm{St}(n,p))$ denote the set of vectors obtained by applying this operation on all the matrices of $\mathrm{St}(n,p)$. We also introduce the notation $\mathbb{S}_r^k$ for the hypersphere of dimension $k$, of radius $r$, centered at the origin. For all $U\in \mathrm{St}(n,p)$, since $\|U\|_{\mathrm{F}} = \|\vec(U)\|_2=\sqrt{p}$, it holds that $\vec(\mathrm{St}(n,p))\subseteq \mathbb{S}_{\sqrt{p}}^{np-1}$. 
\subsection{The lower bound on the {$\beta$}-distance}\label{sec:lowerbound}
We recall a useful property from trigonometry to compute the length of an arc on the hypersphere given the chord length. The property is obtained by working in the triangle $\m{\Delta}$ in~\cref{fig:trigo_picture}.

\begin{prop}\label{prop:vecdist}
    \it
    For all $\m{v}_1,\m{v}_2\in \mathbb{S}_r^k$, the Euclidean geodesic distance (arc-length) is given by $d_{\mathrm{E}}(\m{v}_1,\m{v}_2) = 2r\arcsin\left(\frac{\|\m{v}_1-\m{v}_2\|_2}{2r}\right)$.
\end{prop}

\begin{figure}
    \centering
    \includegraphics[width = 7cm]{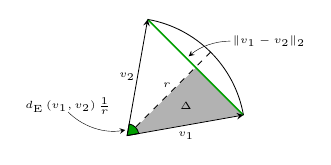}
    \caption{\cref{prop:vecdist} is obtained in the triangle $\m{\Delta}$. It can be deduced from the definition of the sine of the angle in the bottom left corner, which is expressed as the ratio of the length of the side opposite to that angle to the length of the hypotenuse.}
    \label{fig:trigo_picture}
\end{figure}

The next lemma establishes the equivalence of the Euclidean lengths in $\mathrm{St}(n,p)$ and $\vec(\mathrm{St}(n,p))$. This comes from the fact that $\vec:\mathbb{R}^{n\times p}\mapsto\mathbb{R}^{np}$ is an isomorphism when the Euclidean metric is considered in the ambient space $\mathbb{R}^{n\times p}$.

\begin{lem}\label{lem:equivalentlength}
    \it
    The length of a continuously differentiable curve $\gamma:[0,1]\rightarrow \mathrm{St}(n,p)$ under the Euclidean metric, written $l_{\mathrm{E}}(\gamma)$, is equal to the Euclidean length of $\vec(\gamma)$ in $\mathbb{S}_{\sqrt{p}}^{np-1}$, written $l_{\mathrm{E}}(\vec(\gamma))$.
\end{lem}

\begin{proof}
$l_{\mathrm{E}}(\gamma)=\int_0^1\|\Dot{\gamma}(t)\|_{\mathrm{E}}\ \mathrm{d}t =   \int_0^1\|\vec(\Dot{\gamma}(t))\|_2\ \mathrm{d}t = l_{\mathrm{E}}(\vec(\gamma)).$
\end{proof}

\cref{thm:Euclideanlowerbound} is a direct consequence of the fact that $\mathrm{vec}\left(\mathrm{St}(n,p)\right)$ is a subset of $\mathbb{S}_{\sqrt{p}}^{np-1}$: the Euclidean shortest path between any two points in $\mathrm{vec}\left(\mathrm{St}(n,p)\right)$ cannot be shorter than the shortest path in $\mathbb{S}_{\sqrt{p}}^{np-1}$, the last one being the well known arc-length between the two points on the sphere.

\begin{thm}\label{thm:Euclideanlowerbound}
    \it
    For all $U,\widetilde{U}\in \mathrm{St}(n,p)$, we have $d_{\mathrm{E}}(U,\widetilde{U})\geq 2\sqrt{p}\arcsin\left(\frac{\|U-\widetilde{U}\|_{\mathrm{F}}}{2\sqrt{p}}\right)$.
\end{thm}    

\begin{proof}
Using~\cref{prop:vecdist}, we have $d_{\mathrm{E}}(\vec(U),\vec(\widetilde{U}))=2\sqrt{p}\arcsin\left(\frac{\|U-\widetilde{U}\|_{\mathrm{F}}}{2\sqrt{p}}\right)$ since $\|\vec(U)-\vec(\widetilde{U})\|_2 = \|U-\widetilde{U}\|_{\mathrm{F}}$. Moreover, we have
\begin{align*}
    d_{\mathrm{E}}\left(\vec(U),\vec(\widetilde{U})\right) &= \min \left\{l_{\mathrm{E}}(\vec(\gamma)) \ |\ \vec(\gamma(t))\in \mathbb{S}_{\sqrt{p}}^{np},\ \gamma(0)=U,\ \gamma(1)=\widetilde{U}\right\}\\
    &\leq \min \left\{l_{\mathrm{E}}(\vec(\gamma)) \ |\ \gamma(t)\in \mathrm{St}(n,p),\gamma(0)=U, \gamma(1)=\widetilde{U}\right\}\\
    &= \min \left\{l_{\mathrm{E}}(\gamma)\ | \ \gamma(t)\in \mathrm{St}(n,p),\gamma(0)=U,\gamma(1) = \widetilde{U}\right\} \text{(\cref{lem:equivalentlength})}\\
    &=d_{\mathrm{E}}(U,\widetilde{U}).
\end{align*}
\end{proof}

Using the equivalence of the $\beta$-distances, we can generalize~\cref{thm:Euclideanlowerbound} to~\cref{cor:betalowerbound}.

\begin{cor}\label{cor:betalowerbound}
    \it
    For all $U,\widetilde{U}\in \mathrm{St}(n,p)$, we have $$d_{\beta}(U,\widetilde{U})\geq \min\{1,\sqrt{\beta}\}2\sqrt{p}\arcsin\left(\frac{\|U-\widetilde{U}\|_{\mathrm{F}}}{2\sqrt{p}}\right).$$
\end{cor}

\begin{proof}
    Combine~\cref{cor:allbounds} on the bilipschitz equivalence and~\cref{thm:Euclideanlowerbound}.
\end{proof} 

The lower bound of~\cref{cor:betalowerbound} is the lower bound that was drawn by the solid black line on~\cref{fig:shootingfig}. In view of this result, we define
\begin{equation}
\label{eq:m_hat_beta}
\widehat{m}_\beta(\delta):[0,2\sqrt{p}]\rightarrow\mathbb{R}:\delta\mapsto\min\{1,\sqrt{\beta}\}2\sqrt{p}\arcsin\left(\frac{\delta}{2\sqrt{p}}\right).
\end{equation}
\subsection{Attaining the lower bound on the $\beta$-distance}\label{sec:attaining_lower_bound}
To push further this analysis, we find when the bound of~\cref{cor:betalowerbound} is attained, i.e., we find the pairs $n,p$ such that $\widehat{m}_\beta=m_\beta$. \cref{fig:tree} presents a tree summarizing the conditions to decide if $\widehat{m}_\beta = m_\beta$. The sudden change of condition when $\beta$ continuously varies from $\beta<1$ to $\beta>1$ could be surprising. As detailed next, this is due to a change of the family of curves reaching $m_\beta$ at $\beta=1$.
 \begin{figure}[h]
 \centering
 \includegraphics[width = 7cm]{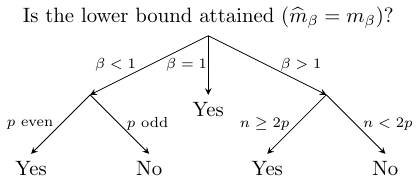}
 \caption{The tree summarizes the conditions to decide whether the lower bound from~\cref{cor:betalowerbound} is attained or not. The conditions change at the critical value $\beta=1$.}
 \label{fig:tree}
 \end{figure}
 
\textbf{Case 1: $p$ even and $\beta\in(0,1]$.} Consider the $2$-by-$2$ rotation matrix $\m{B}_\theta := \left[\begin{smallmatrix}
\cos(\theta)& \sin(\theta)\\
-\sin(\theta)&\cos(\theta)\\
\end{smallmatrix}\right]$ and the block diagonal matrix $\m{G}_p(\theta)$ where
\begin{equation}\label{eq:G_theta}
    \m{G}_p(\theta) :=\begin{bmatrix}
    \m{B}_\theta&0&0&...&0&0\\
    0 & \m{B}_\theta &0&...&0&0\\
    ...&...&...&...&...&...\\
    0&0&0&...&0&\m{B}_\theta\\
    \end{bmatrix}\in \mathrm{SO}(p) \text{ for all } \theta\in \mathbb{R}.
\end{equation}
For $p$ even and $\beta \in (0,1]$,~\cref{thm:lowerboundreached} shows that matrices of the form $\m{UG}_p(\theta)$ attain the lower bound on the geodesic distance to $U$, and thus, $\widehat{m}_\beta=m_\beta$.

\begin{thm}\label{thm:lowerboundreached}
    \it
    Let $p$ be even and $\beta\in (0,1]$. For all $\delta\in[0,2\sqrt{p}]$, there is $\theta \in \mathbb{R}$ such that for all $U\in \mathrm{St}(n,p)$ and all $Q\in\mathrm{O}(p)$,  $d_\beta(U,UQG_{p}(\theta) Q^T)=2\sqrt{\beta p}\arcsin\left(\frac{\delta}{2\sqrt{p}}\right)$ and $\|U-UQG_{p}(\theta) Q^T\|_\mathrm{F} =\delta$. Hence, $\widehat{m}_\beta=m_\beta$. 
\end{thm}

\begin{proof}
Consider $\theta\in [-\pi,\pi]$ since $\theta\mapsto G_{p}(\theta)$ has period $2\pi$, and $\m{\Delta} = UQ A_\theta Q^T$ where \begin{equation}\label{eq:Atheta}
    A_p(\theta):=\begin{bmatrix}
    \m{\Omega}_\theta&0&0&...&0&0\\
    0 & \m{\Omega}_\theta &0&...&0&0\\
    ...&...&...&...&...&...\\
    0&0&0&...&0&\m{\Omega}_\theta\\
    \end{bmatrix}\in \mathrm{Skew}(p)\text{ with } \m{\Omega}_\theta:=\begin{bmatrix}
    0&\theta\\
    -\theta&0
\end{bmatrix}. 
\end{equation}
For such $\Delta$,~\eqref{eq:paramgeogeneral} reduces to $\mathrm{Exp}_{\beta,U}(t\Delta) = UQ\exp\left(A_p(\theta t) \right)Q^T=UQG_p(\theta t)Q^T=:\gamma(t)$. This yields $\Dot{\gamma}(t) = UQ\exp\left(A_p(\theta t)\right)A_p(\theta)Q^T$ and $\|\Dot{\gamma}(t)\|_\beta^2 = \beta\|\m{A}_\theta\|_{\mathrm{F}}^2 = \beta p \theta^2$. It follows from~\eqref{eq:length} that $l_\beta(\gamma) = \sqrt{\beta p}|\theta|$. Moreover, we have
 \begin{equation*}
 	\|\m{UQG}_p(\theta) Q^T-U\|_{\mathrm{F}}^2=\|G_{p}(\theta)-\m{I}\|_{\mathrm{F}}^2=p [(\cos(\theta)-1)^2+\sin(\theta)^2]=4p\sin^2\left(\frac{\theta}{2}\right).
 \end{equation*}
 This yields $|\theta| = 2 \arcsin\left(\frac{\|U-U Q G_{p}(\theta) Q^T\|_{\mathrm{F}}}{2\sqrt{p}}\right)$ from which it follows that $l_\beta(\gamma) = 2\sqrt{\beta p}\arcsin\left(\frac{\|U-UQG_{p}(\theta) Q^T\|_{\mathrm{F}}}{2\sqrt{p}}\right)$. In view of~\cref{cor:betalowerbound}, $\gamma$ is thus a \emph{minimal} geodesic. Hence, we have
\begin{align*}
    d_\beta(U,\m{UQG}_p(\theta) Q^T) =l_\beta(\gamma) = 2\sqrt{\beta p}\arcsin\left(\frac{\|U-\m{UQG}_p(\theta) Q^T\|_{\mathrm{F}}}{2\sqrt{p}}\right).
\end{align*}
Using the intermediate value theorem on $[0,\pi]\ni\theta\mapsto\|U - \m{UQG}_p(\theta) Q^T\|_\mathrm{F}$ shows that $\|U - \m{UQG}_p(\theta) Q^T\|_\mathrm{F}$ attains every value in $[0,2\sqrt{p}]$. This concludes the proof.
\end{proof}

\textbf{Case 2: $n\geq 2p$ and $\beta\geq 1$.} Using a very similar reasoning to~\cref{thm:lowerboundreached}, we also prove that the lower bound is attained for $\beta\geq 1$ if $n\geq 2p$. Let $\theta\in\mathbb{R}$ and define $$K(\theta):=\begin{bmatrix}
\cos(\theta I_p)&\sin(\theta I_p)\\
-\sin(\theta I_p)&\cos(\theta I_p)
\end{bmatrix}=\exp\left(\begin{bmatrix}0&\theta I_p\\
-\theta I_p&0\end{bmatrix}\right)\in\mathrm{SO}(2p).$$
\begin{thm}\label{thm:lowerboundreachedEuclideanetal}
\it
 	Let $n\geq 2p$ and $\beta\geq 1$. For all $\delta\in[0,2\sqrt{p}]$, there is $\theta\in\mathbb{R}$ such that for all $U,\widehat{U}\in\mathrm{St}(n,p)$ with $\widehat{U}^TU=0$ and all $Q\in\mathrm{O}(2p)$, $d_\beta(U,[U\ \widehat{U}]QK(\theta)   Q^TI_{2p\times p})=2\sqrt{p}\arcsin\left(\frac{\delta}{2\sqrt{p}}\right)$ and $\|U - [U\ \widehat{U}]QK(\theta) Q^T I_{2p\times p}\|_\mathrm{F}=\delta$. Hence, $\widehat{m}_\beta=m_\beta$.
\end{thm}
\begin{proof}
Consider $\theta\in [-\pi,\pi]$ since $\theta\mapsto K(\theta)$ has period $2\pi$, and $\m{\Delta} = \m{\widehat{U}\theta I_p}$. In this case, \eqref{eq:paramgeogeneral} reduces to $\mathrm{Exp}_{\beta,U}(t\Delta)=[U\ \widehat{U}]Q K_{\theta t} Q^T I_{2p\times p}=:\gamma(t)$. This yields $\dot{\gamma}(t) = [U\ \widehat{U}] Q K_{\theta t}\left[\begin{smallmatrix}0&\theta I_p\\
-\theta I_p&0\end{smallmatrix}\right] Q^T I_{2p\times p}$ and $\|\dot{\gamma}(t)\|_\beta^2 = \|\theta I_p\|^2_\mathrm{F} = p\theta^2$. It follows from~\eqref{eq:length} that $l_\beta(\gamma) = \sqrt{ p}|\theta|$. Moreover, similarly to~\cref{thm:lowerboundreached}, we have 
\begin{align*}
\|U - [U\ \widehat{U}] Q K(\theta) Q^T I_{2p\times p}\|_\mathrm{F}^2 &=\|I_{n\times p} -I_{n\times 2p}K(\theta) I_{2p\times p}\|^2_\mathrm{F} \\
&= \|I_p-\cos(\Theta)\|_\mathrm{F}^2+\|\sin(\Theta)\|_\mathrm{F}^2\\
&= 4p\sin^2\left(\frac{\theta}{2}\right).
\end{align*} In view of the similarities with the proof of~\cref{thm:lowerboundreached}, $\gamma$ is a \emph{minimal} geodesic and 
\begin{equation*}
	d_\beta(U, [U\ \widehat{U}] Q K(\theta) Q^T I_{2p\times p}) =l_\beta(\gamma)=  2\sqrt{p}\arcsin\left(\frac{\|U - [U\ \widehat{U}]Q K(\theta) Q^T I_{2p\times p}\|_\mathrm{F}}{2\sqrt{p}}\right).
\end{equation*}
Using the intermediate value theorem on $[0,\pi]\ni\theta\mapsto\|U - [U\ \widehat{U}]Q K(\theta) Q^T I_{2p\times p}\|_\mathrm{F}$ shows that $\|U - [U\ \widehat{U}] Q K(\theta) Q^T I_{2p\times p}\|_\mathrm{F}$ attains every value in $[0,2\sqrt{p}]$. This concludes the proof.
\end{proof}
\begin{figure}
    \centering
    \includegraphics[width = 6.4cm]{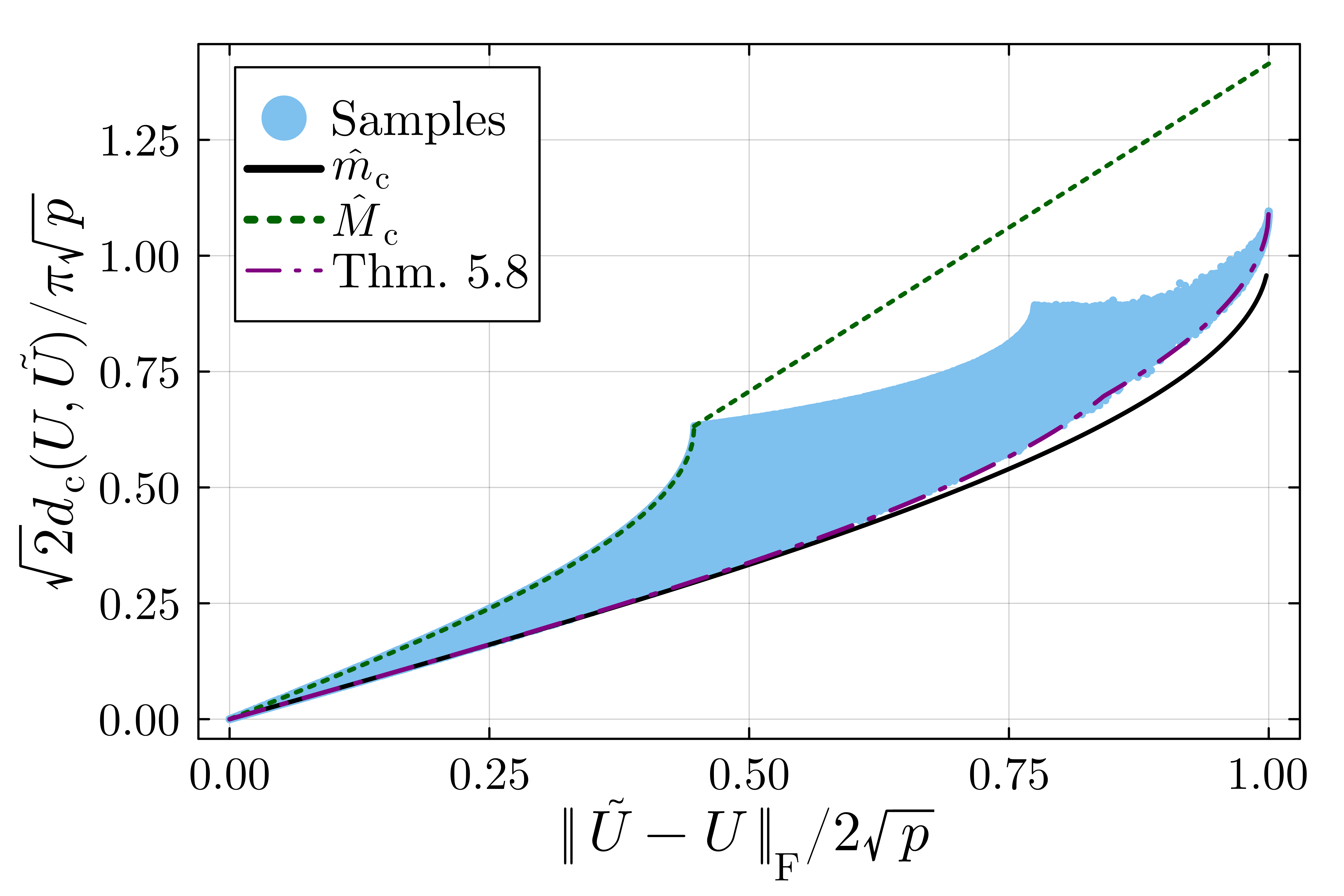}
    \includegraphics[width = 6.4cm]{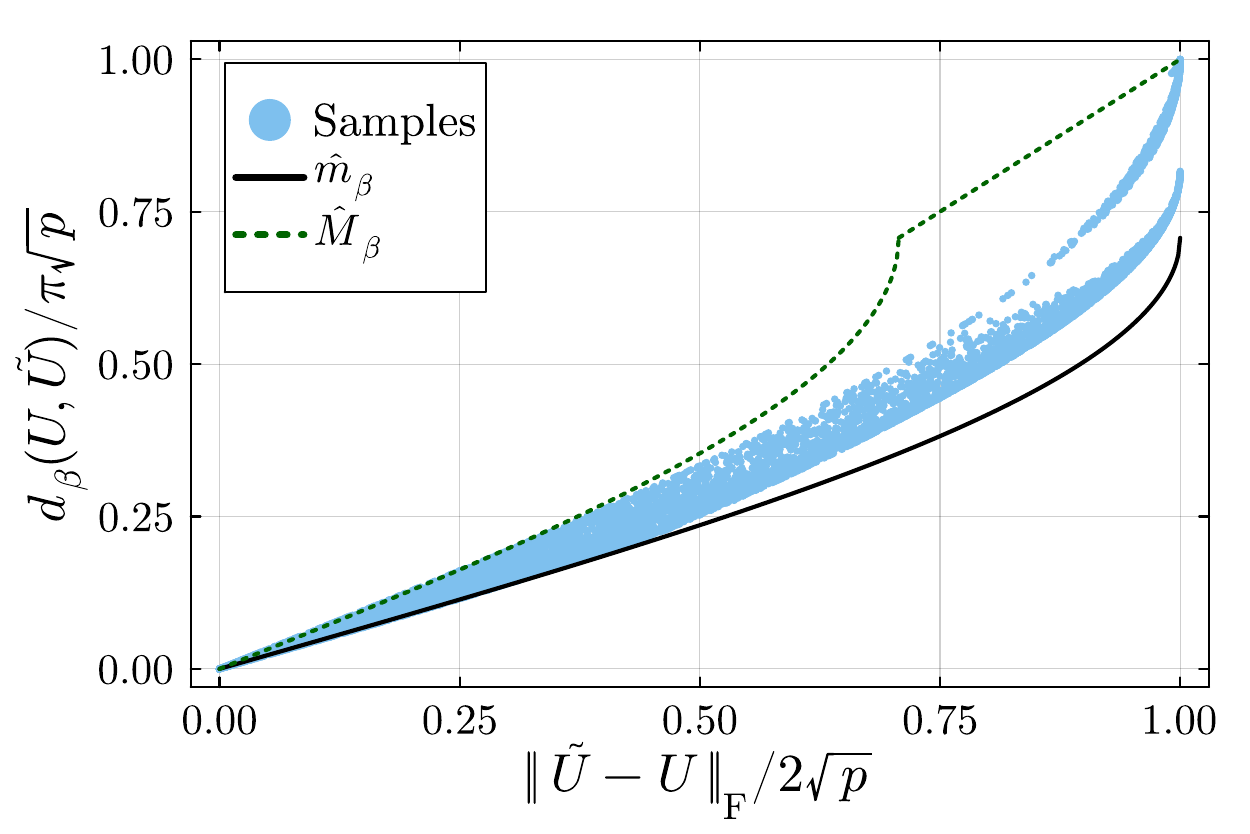}
    \caption{Bounds between the Frobenius and the $\beta$-distance on the Stiefel manifold  for $\beta=\frac{1}{2}$ on $\mathrm{St}(6,5)$ (left) and $\beta=2$ on $\mathrm{St}(3,2)$ (right):  $\widehat{m}_\beta(\delta)\leq d_{\beta}(U,\widetilde{U})\leq \widehat{M}_\beta(\delta)$ for all $U,\widetilde{U}\in \mathrm{St}(n,p)$ with $\|\widetilde{U}-U\|_\mathrm{F} =  \delta$ (see \cref{thm:generalizedbounds}). In both cases, the lower bound is not attained. On the left, we can witness the ability of~\cref{thm:upperboundonthelowerbound} to describe the exact lower bound. On the right, the two observable ``branches'' offer an example of failure of \cite[Alg.~2]{ZimmermannRalf22} to compute minimal geodesics, as explained in~\cref{app:branches}.
    }
    \label{fig:canonicalsamples}
    \end{figure}

	\textbf{Case 3: the Euclidean metric ($\beta=1$).} We anticipate the result of~\cref{cor:reachedforallp} where we show that the lower bound is attained for all Frobenius distances and all pairs $(n,p)$ when $\beta=1$.

    \textbf{Case 4: remaining cases.} The are two remaining cases: $p$ odd with $\beta<1$ and $n<2p$ with $\beta>1$. In both cases, we conjecture that the lower bound is not attained, as corroborated in~\cref{fig:canonicalsamples} with random numerical experiments. In the left plot, we investigate the canonical metric\footnote{When $p=n-1$, the canonical geodesic distance admits an analytic solution: $d_\mathrm{c}(U,\widetilde{U})=\frac{\sqrt{2}}{2}\|\log([U\ U_\perp]^T[\widetilde{U}\ \widetilde{U}_\perp])\|_\mathrm{F}$ where the completed matrices belong to $\mathrm{SO}(n)$ \cite[Eq.~2.15]{EdelmanArias98}. Therefore, each sample in the left plot of~\cref{fig:canonicalsamples} is exactly positioned in the $(d_\mathrm{F},d_\mathrm{c})$-plane (up to the computer precision of algebraic operations).} ($\beta=\frac{1}{2}$) with $(n,p)=(6,5)$ and in the right plot, we draw experiments for $\beta=2$ and $(n,p)=(3,2)$. In both plots, we can observe a gap between the samples and $\widehat{m}_\beta$, suggesting $\widehat{m}_\beta\neq m_\beta$. We propose here to push further the analysis in the case that is certainly the most interesting for practical applications | $p$ odd and $\beta\in[\frac{1}{2},1]$ | because it interpolates the canonical and Euclidean metrics. \cref{thm:upperboundonthelowerbound} will propose an upper bound on $m_\beta$ when $p$ is odd, i.e., a function $w_\beta$ with $m_\beta(\delta)\leq w_\beta(\delta)$ for all $\delta\in[0,2\sqrt{p}]$. To this end,~\cref{thm:increasingdist} provides an intermediate result based on the following observation: when one appends supplementary columns to matrices of the Stiefel manifold, the geodesic distance in the new associated Stiefel manifold increases for $\beta\geq\frac{1}{2}$.
    \begin{thm}\label{thm:increasingdist}
    \it
    If $\beta\geq\frac{1}{2}$, for every $n\geq p\geq 2$ and for all $U,\widetilde{U}\in\mathrm{St}(n,p)$, we have $d_\beta(U,\widetilde{U})\geq d_\beta(U_{1:p-1},\widetilde{U}_{1:p-1})$ and $\beta=\frac{1}{2}$ is the smallest value so that it holds true. 
    \end{thm}
    \begin{proof}
    Given $\Delta = UA+U_\perp B\in\mathrm{Log}_{\beta,U}(\widetilde{U})$, we define $S_\beta:=\left[\begin{smallmatrix}
    2\beta A& -B^T\\
    B &0
    \end{smallmatrix}\right]$ and $R_{\beta}(t):=\left[\begin{smallmatrix}
    \exp((1-2\beta)At)&0\\
    0&I_{n-p}
    \end{smallmatrix}\right]$. Consider the curve $\gamma:[0,1]\rightarrow\mathrm{St}(n,p-1):t\mapsto\mathrm{Exp}_{\beta,U}(t\Delta)I_{n\times(p-1)}=[U\ U_\perp]\exp(tS_\beta)R_{\beta}(t)I_{n\times(p-1)}$. Notice that $\gamma(0)=U_{1:p-1}$ and $\gamma(1)=\widetilde{U}_{1:p-1}$. The strategy is to show that $l_\beta(\gamma)\leq\|\Delta\|_\beta=:d_\beta(U,\widetilde{U})$. 
    In what follows, we leverage that $A$ and $\exp((1-2\beta)At)$ commute. Given that $\dot{\gamma}(t) =[U\ U_\perp]\exp(tS_\beta)S_{\frac{1}{2}}R_{\beta}(t)I_{n\times(p-1)}$ and by definition \eqref{eq:beta_metric} of the $\beta$-metric, we have
    \begin{align*}
    \|\dot{\gamma}(t)\|_\beta^2&=\Tr\left(I_{(p-1)\times n}R_{\beta}^T(S_{\frac{1}{2}}^TS_{\frac{1}{2}}-(1-\beta)S_{\frac{1}{2}}^TR_{\beta}\left[\begin{smallmatrix}
    I_{p-1}&0\\
    0&0_{n-p+1}
    \end{smallmatrix}\right]R_{\beta}^TS_{\frac{1}{2}})R_{\beta}I_{n\times (p-1)}\right)\\
    &=\Tr\left(I_{p-1\times p}\exp((1-2\beta)At)^T(\beta A^TA+B^TB)\exp((1-2\beta)At)I_{p\times p-1}\right)\\
    &+\Tr(I_{p-1\times p}((1-\beta)A^T\left[\begin{smallmatrix}
    0_{p-1}&0&0\\
    0&1&0\\
    0&0&0_{n-p+1}
    \end{smallmatrix}\right]A)I_{p\times p-1})\\
    &=\beta \|A_{1:p,1:p-1}\|_\mathrm{F}^2 + \|B\exp((1-2\beta)At)I_{p\times p-1}\|_\mathrm{F}^2 + (1-\beta)\|A_{p,1:p-1}\|_\mathrm{F}^2\\
    &\leq \beta \|A_{1:p,1:p-1}\|_\mathrm{F}^2 + \|B\|_\mathrm{F}^2 + (1-\beta)\|A_{p,1:p-1}\|_\mathrm{F}^2.
    \end{align*}
    To conclude, we want $\beta$ to satisfy $(1-\beta)\leq\beta$, which holds if and only if $\beta\geq\frac{1}{2}$. This yields
    \begin{align*}
    l_\beta(\gamma) = \int_0^1 \|\dot{\gamma}(t)\|_\beta\mathrm{d}t&\leq\sqrt{\beta \|A_{1:p,1:p-1}\|_\mathrm{F}^2 + \|B\|_\mathrm{F}^2 + (1-\beta)\|A_{p,1:p-1}\|_\mathrm{F}^2}\\
    &\leq \sqrt{\beta (\|A_{1:p,1:p-1}\|_\mathrm{F}^2+\|A_{p,1:p-1}\|_\mathrm{F}^2) + \|B\|_\mathrm{F}^2}\\
    &=\sqrt{\beta \|A\|_\mathrm{F}^2 + \|B\|_\mathrm{F}^2}=\|\Delta\|_\beta=d_\beta(U,\widetilde{U}).
    \end{align*}
    By definition, we have $d_\beta(U_{1:p-1},\widetilde{U}_{1:p-1})\leq l_\beta(\gamma)\leq d_\beta(U,\widetilde{U})$. Finally, assume $\beta<\frac{1}{2}$, take $U=I_{3\times 2}$ and $\widetilde{U}=-U$ in $\mathrm{St}(3,2)$. Consider \cref{thm:lowerboundreached} with $Q=I_p$ and $\delta=2\sqrt{p}$. We must have $\mod(|\theta|, 2\pi)=\pi$ for $UQG_p(\pi)Q^T=-U$ to hold. Consequently, $d_\beta(U,-U)=2\sqrt{2\beta }\arcsin(1)=\pi\sqrt{2\beta}<\pi$. However, regardless of $\beta$, we have $d_\beta(U_{1},\widetilde{U}_1)=\pi$, i.e., the geodesic distance between two antipodal points on the unit sphere in the Euclidean space $\mathbb{R}^n$. Hence, $\beta=\frac{1}{2}$ is the smallest value such that~\cref{thm:increasingdist} holds.
    \end{proof}
    
    We leverage~\cref{thm:increasingdist} to prove that, although $\widehat{m}_\beta\neq m_\beta$ for $p$ odd, we can bound $m_\beta$ from above. 
    \begin{thm}\label{thm:upperboundonthelowerbound}
    \it
    If $p$ is odd and $n>p$, for all $\beta\in[\frac{1}{2},1]$, all $\delta\in[0,2\sqrt{p}]$, $$m_\beta(\delta)\leq 2\sqrt{\beta}\min\left\{\sqrt{p-1}\arcsin\left(\frac{\delta}{2\sqrt{p-1}}\right),\sqrt{p+1}\arcsin\left(\frac{\delta}{2\sqrt{p}}\right)\right\},$$
     where the first argument of the $\min$ is only considered if $p\geq3 $ and $\delta\leq2\sqrt{p-1}$.
    \end{thm}
    \begin{proof}
    We start by considering the first argument of the $\min$. For well-definedness of $G_{p-1}(\theta)$, it requires $p\geq3$. For all $\theta\in [-\pi,\pi]$, consider $U\in\mathrm{St}(n,p)$ and $\widetilde{U}_\theta=[U_{1:p-1}G_{p-1}(\theta)\ U_p]$ with $G_{p-1}(\theta)\in\mathrm{SO}(p-1)$ (see \eqref{eq:G_theta}). By~\cref{thm:lowerboundreached}, the curve $\gamma:[0,1]\rightarrow\mathrm{St}(n,p):t\mapsto\widetilde{U}_{\theta t}$ has length $l_\beta(\gamma)=2\sqrt{\beta(p-1)}\arcsin\left(\frac{\|U-\widetilde{U}_\theta\|_\mathrm{F}}{2\sqrt{p-1}}\right)$. Since $d_\beta(U,\widetilde{U}_\theta)\leq l_\beta(\gamma)$ and given that $[0,\pi]\ni\theta\mapsto\|U-\widetilde{U}_{\theta}\|_\mathrm{F}$ achieves every value in $[0,2\sqrt{p-1}]$, the first part of the statement is proven.
    
    We address now the second argument of the claim. For all $U,\widetilde{U}\in\mathrm{St}(n,p)$, all $u_\perp,\widetilde{u}_\perp\in\mathrm{St}(n,1)$ with $u_\perp^TU=\widetilde{u}_\perp^T\widetilde{U}=0$, we know from~\cref{thm:increasingdist} that
    \begin{align*}
    d_\beta(U,\widetilde{U})&\leq d_\beta([U\ u_\perp],[\widetilde{U}\ \widetilde{u}_\perp])\\
    \Longrightarrow m_\beta(\delta):= \min_{\|U-\widetilde{U}\|_\mathrm{F}=\delta}d_\beta(U,\widetilde{U})&\leq\min_{\|U-\widetilde{U}\|_\mathrm{F}=\delta}d_\beta([U\ u_\perp],[\widetilde{U}\ \widetilde{u}_\perp]).
    \end{align*}
    \cref{thm:lowerboundreached} shows that the minimum value of the right-hand-side is attained for $[\widetilde{U}\ \widetilde{u}_\perp]_{\theta} := [U\ u_\perp]G_{p+1}(\theta)$ with $G_{p+1}(\theta)\in\mathrm{SO}(p+1)$ for some $\theta$ such that $\|U-\widetilde{U}\|_\mathrm{F}=\delta$. It is clear that $\|U-\widetilde{U}\|_\mathrm{F}=0$  when $\theta=0$ and $\|U-\widetilde{U}\|_\mathrm{F}=2\sqrt{p}$ when $\theta=\pi$. Hence, the intermediate value theorem ensures the existence of $\theta$ for all $\delta\in[0,2\sqrt{p}]$. We relate $\delta$ to $\theta$ precisely:
    \begin{align*}
    \delta^2=\|U-[U\ u_\perp]G_{p+1}(\theta)I_{(p+1)\times p}\|_\mathrm{F}^2 = 4p\sin^2\left(\frac{\theta}{2}\right)\Longrightarrow |\theta| &= 2\arcsin\left(\frac{\delta}{2\sqrt{p}}\right).
    \end{align*}
    Moreover, $\|[U\ u_\perp]-[\widetilde{U}\ \widetilde{u}_\perp]_{\theta}\|_\mathrm{F} = 2\sqrt{(p+1)\sin^2(\frac{\theta}{2})}$ and by~\cref{thm:lowerboundreached}, we have
    \begin{align*}
    d_\beta([U\ u_\perp],[\widetilde{U}\ \widetilde{u}_\perp]_{\theta})&=2\sqrt{p+1}\arcsin\left(\frac{\|[U\ u_\perp]-[\widetilde{U}\ \widetilde{u}_\perp]_{\theta}\|_\mathrm{F}}{2\sqrt{p+1}}\right)\\
    &=2\sqrt{p+1}\arcsin\left(\frac{2\sqrt{(p+1)\sin^2(\frac{\theta}{2})}}{2\sqrt{p+1}}\right)\\
    &=2\sqrt{p+1}\arcsin\left(\frac{\delta}{2\sqrt{p}}\right).
    \end{align*}
    In consequence, we showed that
    $m_\beta(\delta)\leq 2\sqrt{p+1}\arcsin\left(\frac{\delta}{2\sqrt{p}}\right)$.
    This concludes the proof.
    \end{proof}
    \cref{thm:upperboundonthelowerbound} concludes a detailed study of the smallest geodesic distance given the Frobenius distance. In the next sections, we address the largest geodesic distance given the Frobenius distance.

\section{Upper bounds on the Euclidean geodesic distance by the Frobenius distance} \label{sec:general_upper_bound}
In this section, we obtain two upper bounds on the Euclidean geodesic distance. The results will be extended to all $\beta$-distances in~\cref{sec:generalization_to_beta}. We start by proving a linear upper bound in terms of the Frobenius distance in~\cref{thm:linearbound}. As corroborated by~\cref{fig:error_bound}, this bound is a good fit to $M_\mathrm{E}$ for pairs of matrices $U,\widetilde{U}$ with $\|U-\widetilde{U}\|_\mathrm{F}> 2$ since the relative error is less than $10\%$. However, for nearby pairs of matrices, the relative error of the linear upper bound raises to $57\%$. Therefore, when $\|U-\widetilde{U}\|_\mathrm{F}\leq 2$, we obtain in~\cref{thm:first_upper_bound} a nonlinear upper bound in terms of the Frobenius distances. As shown by~\cref{cor:reach_first_upper_bound}, this nonlinear bound cannot be improved, i.e., $\widehat{M}_\mathrm{E}(\delta)=M_\mathrm{E}(\delta)$ for all $\delta\in[0,2]$. 
\begin{figure}
\centering
\includegraphics[width = 7cm]{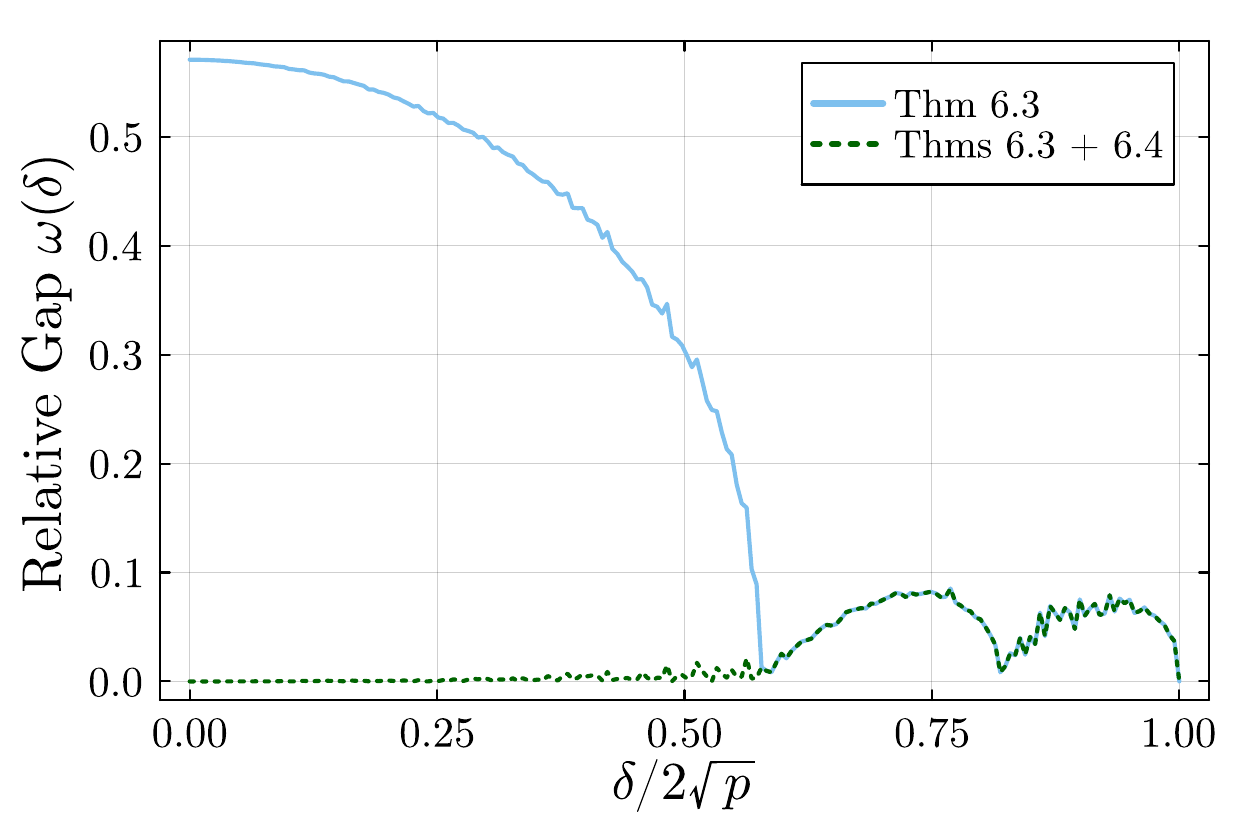}
\caption{Relative error between the theoretical upper bound (considering either only~\cref{thm:linearbound} or including~\cref{thm:first_upper_bound}) and the largest Euclidean geodesic distance obtained by numerical experiments on $\mathrm{St}(4,3)$. Let $\delta\in[0,2\sqrt{p}]$, $M_\mathrm{E}(\delta) := \max_{\|U-\widetilde{U}\|_\mathrm{F}=\delta}d_\mathrm{E}(U,\widetilde{U})$ and $\widehat{M}_\mathrm{E}(\delta)$ be the upper bound. Then, the relative gap is defined by $\omega(\delta) :=\frac{\widehat{M}_\mathrm{E} - M_\mathrm{E}(\delta)}{M_\mathrm{E}(\delta)}$. The value $M_\mathrm{E}(\delta)$ is estimated from random samples $U_i,\widetilde{U}_i$ by $M_\mathrm{E}(\delta)\approx \max_{\|U_i-\widetilde{U}_i\|_\mathrm{F}\in[\delta+\frac{\sqrt{p}}{100})}d_\mathrm{E}(U_i,\widetilde{U}_i)$ using \cite[Alg.~2]{ZimmermannRalf22} to estimate $d_\mathrm{E}(U_i,\widetilde{U}_i)$.}
\label{fig:error_bound}
\end{figure}
\subsection{Geometric sets in the ambient space $\mathbb{R}^{n\times p}$} \label{sec:sets}
Specific curves are needed to continue exploring the bounds from~\cref{fig:shootingfig}. These curves are related to the sets introduced in this section, illustrated in~\cref{fig:sets}. Given any two points $U,\widetilde{U}\in \mathrm{St}(n,p)$, we define the matrix-hypersphere
\begin{equation}\label{eq:bigsphere}
    \mathbb{S}_{\sqrt{p}}^{n\times p}:=\left\{\m{X}\in\mathbb{R}^{n\times p}\ | \ \langle X,X\rangle_\mathrm{F}  = p\right\}, 
\end{equation}
and a second matrix-hypersphere $\Sigma_{U,\widetilde{U}}$ for which $U$ and $\widetilde{U}$ are antipodal points. Let $\m{C}=\frac{U+\widetilde{U}}{2}$ and $r = \frac{\|\widetilde{U}-U\|_{\mathrm{F}}}{2}$. The matrix-hypersphere $\Sigma_{U,\widetilde{U}}$ is defined by 
\begin{equation}\label{eq:littlesphere}
    \Sigma_{U,\widetilde{U}}:=\left\{\m{X}\in\mathbb{R}^{n\times p}\ | \ \langle X-C,X-C \rangle_\mathrm{F} = r^2\right\}.
\end{equation}
If $\widetilde{U}\neq-U$, the intersection of~\eqref{eq:bigsphere} and~\eqref{eq:littlesphere} belongs to the matrix-hyperplane $\Pi_{U,\widetilde{U}}$ defined by the set
\begin{equation}\label{eq:planeofcap}
    \Pi_{U,\widetilde{U}}:= \left\{\m{X}\in \mathbb{R}^{n\times p}\ | \ \langle X,C\rangle_\mathrm{F} = \langle U,C\rangle_\mathrm{F}\right\}.
\end{equation}
The sets $\mathbb{S}_{\sqrt{p}}^{n\times p},\Sigma_{U,\widetilde{U}}$ and $\Pi_{U,\widetilde{U}}$ are essential to obtain geometric insights on $\mathrm{St}(n,p)$ and are at the core of~\cref{thm:Euclideanlowerbound},~\cref{thm:linearbound} and~\cref{thm:reachlinearbound}. 
Given the absence of any ambiguity, we will simply refer to these sets as ``hypersphere/plane'' instead of ``matrix-hypersphere/plane''.
\begin{figure}
    \centering
    \includegraphics[width = 7cm]{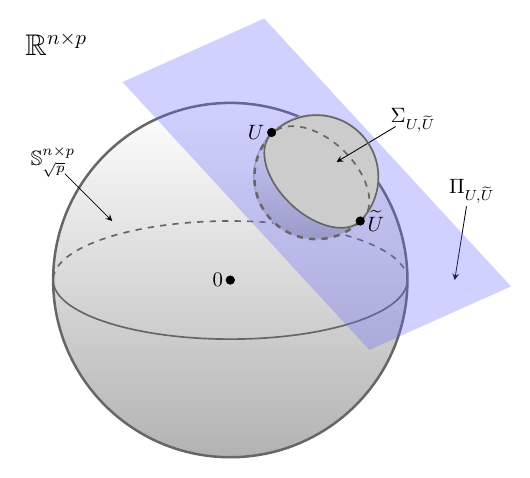}
    \caption{Abstract representation (correct for $n=3$, $p=1$) of the sets $\mathbb{S}_{\sqrt{p}}^{n\times p},\Sigma_{U,\widetilde{U}}$ and $\Pi_{U,\widetilde{U}}$ introduced in~\cref{sec:sets}. The hypersphere $\mathbb{S}_{\sqrt{p}}^{n\times p}$, which contains the Stiefel manifold $\mathrm{St}(n,p)$, also contains the points $U,\widetilde{U}\in\mathrm{St}(n,p)$. $U,\widetilde{U}$ are antipodal points of $\Sigma_{U,\widetilde{U}}$ and $\Pi_{U,\widetilde{U}}$ is the hyperplane that contains the intersection of $\mathbb{S}_{\sqrt{p}}^{n\times p}$ and $\Sigma_{U,\widetilde{U}}$.}
    \label{fig:sets}
\end{figure}

\subsection{A linear upper bound on the Euclidean geodesic distance by the Frobenius distance} \label{sec:upperbound} 
We wish to obtain a linear upper bound on the $\beta$-distance by the Frobenius distance. Two milestones (\cref{lem:augmentedrows} and~\cref{lem:condforsyst}) are needed to establish $d_\mathrm{E}(U,\widetilde{U})\leq \frac{\pi}{2}\|\widetilde{U}-U\|_{\mathrm{F}}$ (\cref{thm:linearbound}). For all $U\in \mathrm{St}(n,p)$, we define the augmentation of $U$, $U_{N}\in \mathrm{St}(n+N,p)$, as the matrix obtained by appending $N$ rows of zeros to $U$: $U_N := \begin{bmatrix}
    U\\
    \m{0}_{N\times p}
\end{bmatrix}$.~\cref{lem:augmentedrows} shows that it leaves the geodesic distance unchanged when $n\geq 2p$ for all $\beta>0$. If $\beta\leq 1$, we conjecture that $n\geq 2p$ is a ``parasite condition'', i.e.,~\cref{lem:augmentedrows} holds true for $p\leq n<2p$ (when $p=n$, the two connected of $\mathrm{St}(n,n)=\mathrm{O}(n)$ must be considered separately for the distance to be well-defined). Although we could not prove this conjecture, it is important to keep in mind that most applications from machine learning and statistics on manifolds feature a ``$n\gg p$'' setting. If $\beta>1$, counter-examples show that~\cref{lem:augmentedrows} does not hold for $n<2p$, see \cref{app:branches}.

\begin{lem}\label{lem:augmentedrows}
    \it
    Let $n\geq 2p$. For every $\beta>0$, every $N\in\mathbb{N}$ and all $U,\widetilde{U}\in \mathrm{St}(n,p)$, it holds that $d_{\beta}(U,\widetilde{U})=d_{\beta}(U_{N},\widetilde{U}_{N})$.
\end{lem}
\begin{proof}
We start from the definition of geodesic distance on $\mathrm{St}(n+N,p)$:
\begin{equation}\label{eq:distproblem}
	d_\beta(U_N,\widetilde{U}_N) = \min\{\|\Delta\|_\beta\ |\ \Delta\in \mathrm{T}_{U_N} \mathrm{St}(n+N,p),\ \mathrm{Exp}_{\beta,U_N}(\Delta) = \widetilde{U}_N\}.
\end{equation}
Let $\Delta$ realize the $\min$ in~\eqref{eq:distproblem}. Then  $\Delta = U_N A + QB$ with $ Q\in\mathrm{St}(n+N,p)$ and $Q^TU_N=0$. 
In view of \eqref{eq:paramgeogeneral}, it holds that $\mathrm{col}([U_N\ Q])\supseteq \mathrm{col}([U_N\ \widetilde{U}_N])$. Hence, there exist $R\in\mathrm{O}(p)$ such that $QR:=[\widehat{Q}_{1}\ \widehat{Q}_{2}]$ 
satisfies \begin{equation}\label{eq:conservecolumns}
\mathrm{col}([U_N\ \widehat{Q}_{1}])=\mathrm{col}([U_N\ \widetilde{U}_N]).
\end{equation}
Let $q$ denote the number of columns of $\widehat{Q}_{1}$. If $q=0$, i.e., $\mathrm{col}(U_N)=\mathrm{col}(\widetilde{U}_N)$, $\widehat{Q}_{1}$ disappears. If $q=p$, i.e., $\mathrm{col}(U_N)\cap\mathrm{col}(\widetilde{U}_N)=\emptyset$, $\widehat{Q}_{2}$ disappears. In both cases, the rest of the proof still holds, mutatis mutandis. 
In view of \eqref{eq:paramgeogeneral}, $\mathrm{Exp}_{\beta,U_N}(\Delta)$ holds if and only if
\begin{equation}
	\label{eq:Q2vanishing}
	\begin{bmatrix}
	U_N^T\widetilde{U}_N\\
	\widehat{Q}_1^T\widetilde{U}_N\\
	\widehat{Q}_2^T\widetilde{U}_N\\
	\end{bmatrix} = \exp\begin{bmatrix}
	2\beta A&-B^T R\\
	R^TB&0
	\end{bmatrix}I_{2p\times p}\exp((1-2\beta)A).
\end{equation}
By \eqref{eq:conservecolumns}, it holds that $\widehat{Q}_{1}=\left[\begin{smallmatrix}
\breve{Q}_1\\
0_{N\times p}
\end{smallmatrix}\right]$. 
Moreover, $\widehat{Q}_{2}^T U_N = \widehat{Q}_{2}^T \widetilde{U}_N=0$, or equivalently, $\mathrm{col}(\widehat{Q}_{2})\subseteq \mathrm{col}([U_N\ \widetilde{U}_N])^\perp$. 
Hence, $\eqref{eq:Q2vanishing}$ still holds if $\widehat{Q}_2$ is replaced by any orthogonal basis $\overline{Q}_2$ of $\mathrm{col}([U_N\ \widetilde{U}_N])^\perp$. In particular, since $n\geq 2p$, we can take $\overline{Q}_2 =\left[\begin{smallmatrix}
\breve{Q}_2\\
0_{N\times p-q}
\end{smallmatrix}\right]$.
Therefore, $\left[\begin{smallmatrix}
\breve{\Delta}\\
0_{N\times p}
\end{smallmatrix}\right] := \left[\begin{smallmatrix}
U\\
0_{N\times p}
\end{smallmatrix}\right] A +\left[\begin{smallmatrix}
\breve{Q}_1&\breve{Q}_2\\
0_{N\times q}& 0_{N\times p-q}
\end{smallmatrix}\right]R^TB$ realizes the $\min$ of~\eqref{eq:distproblem}, which implies that 
\begin{align*}
d_\beta(U_N,\widetilde{U}_N) &= \min\{\|\Delta\|_\beta\ |\ \widetilde{\Delta}\in \mathrm{T}_{U_N} \mathrm{St}(n+N,p),\ \mathrm{Exp}_{\beta,U_N}(\Delta) = \widetilde{U}_N, \Delta=\left[\begin{smallmatrix}\breve{\Delta}\\0_{N\times p}\end{smallmatrix}\right] \}\\
&=\min\{\|\breve{\Delta}\|_\beta\ |\ \breve{\Delta} \in \mathrm{T}_{U} \mathrm{St}(n,p),\ \mathrm{Exp}_{\beta,U}(\breve{\Delta}) = \widetilde{U}\}\\
&=:d_{\beta}(U,\widetilde{U}).
\end{align*}
This concludes the proof.
\end{proof}

The condition $n\geq2p$ from \cref{lem:augmentedrows} spreads across \cref{thm:linearbound}, \ref{thm:Euclideandiameter}, \ref{thm:reachlinearbound}, \ref{thm:Euclidean_minimality} and~\ref{thm:generalizedbounds}. This condition could be neglected if \cref{conj:rows_beta} holds.
\begin{conj}\label{conj:rows_beta}
If $\beta\leq 1$, \cref{lem:augmentedrows} holds for all pairs $(n,p)$ with $n>p$ and, if $n=p$, on each connected component of $\mathrm{St}(n,n)$ separately.
\end{conj}

The second milestone highlights the existence of a solution to the nonlinear system of matrix equations~\eqref{eq:system}. These equations are essential conditions arising in the proof of~\cref{thm:linearbound}.

\begin{lem}\label{lem:condforsyst}
    \it
    Let $n\geq 3p$. For all $U,\widetilde{U}\in \mathrm{St}(n,p)$ and given $\m{R}:=\frac{\widetilde{U}-U}{2}$, the system 
    \begin{equation}\label{eq:system}
        \begin{cases}
            U^T\m{X}+\m{X}^TU= \m{0}\\
            \widetilde{U}^T\m{X}+ \m{X}^T\widetilde{U}= \m{0}\\
            \m{X}^T\m{X} =\m{R}^T\m{R}
        \end{cases} 
    \end{equation} admits at least a solution $\m{X}\in \mathbb{R}^{n\times p}$.
\end{lem}

\begin{proof}
    Since $n\geq3p$, there exists $\widehat{U}\in \mathrm{St}(n,p)$ such that $\widehat{U}^TU=\widehat{U}^T\widetilde{U}=\m{0}$. For all $\m{Y}\in\mathbb{R}^{p\times p}$,  $\m{X}:=\widehat{U}\m{Y}$ satisfies the first two equations. The third equation yields $\m{Y}^T\m{Y}=\m{R}^T\m{R}$ which admits at least one solution. For instance, take $\m{Y} = (\m{R}^T\m{R})^{\frac{1}{2}}$, which is a well-defined real solution since $\m{R}^T\m{R}$ is symmetric positive semidefinite.
\end{proof}

Gathering the results of~\cref{lem:augmentedrows} and~\cref{lem:condforsyst}, we can finally prove the upper bound (green dashed line in~\cref{fig:shootingfig}) on the Euclidean geodesic distance by the Frobenius distance in~\cref{thm:linearbound}.
\begin{thm}\label{thm:linearbound}
    \it
    Let $n\geq 2p$. For all $U, \widetilde{U} \in \mathrm{St}(n,p)$, $d_{\mathrm{E}}(U,\widetilde{U})\leq \frac{\pi}{2}\|\widetilde{U}-U\|_{\mathrm{F}}$.
\end{thm}
\begin{proof}
Start by appending $N$ rows of zeros to $U$ and $\widetilde{U}$ such that $n+N\geq 3p$ and redefine $n\leftarrow n+N$ for simplicity. By~\cref{lem:augmentedrows}, the $\beta$-distances are left unchanged and the conclusion of~\cref{lem:condforsyst} holds. For $S\in \mathbb{R}^{n\times p}$, consider the curve $\gamma_S:[0,\pi]\rightarrow\mathbb{R}^{n\times p}$ defined by
\begin{equation}\label{eq:defofgamma}
    \gamma_S(t) := \m{C} + \cos(t)\m{R}+\sin(t)\m{S},
\end{equation}
where $\m{C}=\frac{U+\widetilde{U}}{2}$ and $\m{R} = \frac{U-\widetilde{U}}{2}$. The curve $\gamma_S$ and the geometric intuition underlying its expression are illustrated in~\cref{fig:sphericalcap}. We find sufficient conditions on $\m{S}$ such that $\gamma_S(t)$ lies on $\mathrm{St}(n,p)$ for all $t\in\mathbb{R}$. By definition of $\gamma_S$, we have
\begin{align}
	\label{eq:gammaonstiefel}
    \gamma_S(t)^T\gamma_S(t) &= \m{C}^T\m{C} +\cos(t)(\m{C}^T\m{R}+\m{R}^T\m{C})+\sin(t)(\m{C}^T\m{S}+\m{S}^T\m{C})\\
    \nonumber
    &\quad\quad\quad\quad+\cos^2(t)(\m{R}^T\m{R}) + \sin^2(t)(\m{S}^T\m{S}).
\end{align}
Observe that $C^T R+R^T C=0$, hence the $\cos(t)$ term vanishes. Let $\m{S}\in \mathbb{R}^{n\times p}$ satisfy
\begin{equation}\label{eq:systtosolve}
    \begin{cases}
        \m{C}^T\m{S}+\m{S}^T\m{C}= \m{0}\\
        \m{R}^T\m{S}+ \m{S}^T\m{R}= \m{0}\\
        \m{S}^T\m{S} =\m{R}^T\m{R}
    \end{cases} \text{or equivalently }\begin{cases}
        U^T\m{S}+\m{S}^TU= \m{0}\\
        \widetilde{U}^T\m{S}+ \m{S}^T\widetilde{U}= \m{0}\\
        \m{S}^T\m{S} =\m{R}^T\m{R}
    \end{cases}.
\end{equation}
$S$ exists in view of~\cref{lem:condforsyst}. Then, \eqref{eq:gammaonstiefel} yields $\gamma_S(t)^T\gamma_S(t)=C^TC+R^TR = I_p$, i.e., $\gamma_S(t)\in \mathrm{St}(n,p)$ for all $t\in\mathbb{R}$. 
Moreover, $\gamma_S(t)$ also belongs to the hypersphere $\Sigma_{U,\widetilde{U}}$ (see~\eqref{eq:littlesphere}) for all $t\in\mathbb{R}$. 
Indeed, we have
\begin{align*}
    \langle\gamma_S(t)-\m{C},\gamma_S(t)-\m{C}\rangle_\mathrm{F} &= \Tr\left((\cos(t)\m{R}+\sin(t)\m{S})^T(\cos(t)\m{R}+\sin(t)\m{S})\right)\\
    &=\Tr(\m{R}^T\m{R})\quad\quad\text{Since } \m{R}^T\m{S}+\m{S}^T\m{R} =\m{0}\text{ and } \m{S}^T\m{S}=\m{R}^T\m{R}.\\
    &=\left(\frac{1}{2}\|\widetilde{U}-U\|_{\mathrm{F}}\right)^2.
\end{align*}
In view of \eqref{eq:defofgamma}, $\gamma_S$ draws a half great-circle of Euclidean length $\frac{\pi}{2}\|\widetilde{U}-U\|_{\mathrm{F}}$ (see~\cref{fig:sphericalcap} for the geometric intuition). As an analytical corroboration, we have
\begin{equation*}
    l_{\mathrm{E}}(\gamma_S)=\int_0^\pi \|\dot{\gamma}_S(t)\|_\mathrm{F}
    \ \mathrm{d}t=\int_0^\pi \|\m{R}\|_\mathrm{F} \ \mathrm{d}t=\frac{\pi}{2}\|\widetilde{U}-U\|_{\mathrm{F}}.
\end{equation*}
By definition~\eqref{eq:geodesicdistance} of the distance, we have $d_{\mathrm{E}}(U,\widetilde{U})\leq l_\mathrm{E}(\gamma_S)=\frac{\pi}{2}\|\widetilde{U}-U\|_{\mathrm{F}}$. 
\end{proof}
\begin{figure}
    \centering
    \includegraphics[width = 9cm]{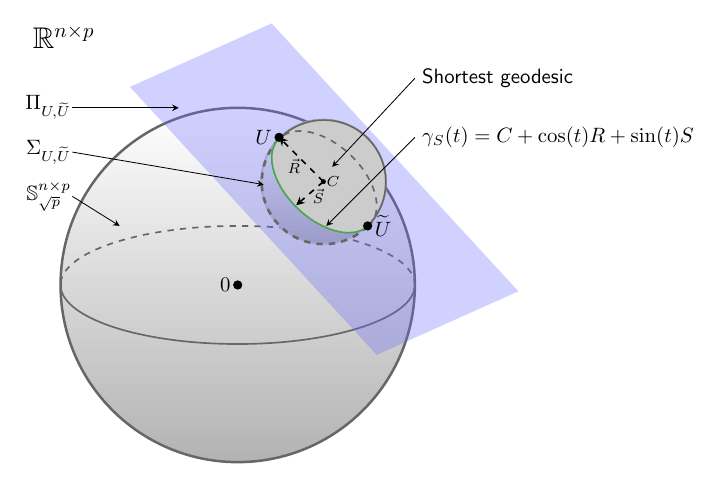}
    \caption{Abstract representation (correct for $n=3$, $p=1$) of the sets $\mathbb{S}_{\sqrt{p}}^{n\times p},\Sigma_{U,\widetilde{U}}$ and $\Pi_{U,\widetilde{U}}$ introduced in~\cref{sec:sets}. The planar curve $\gamma_S$ travelling at constant speed from $U$ to $\widetilde{U}$, of length $\frac{\pi}{2}\|\widetilde{U}-U\|_{\mathrm{F}}$ and remaining in $\mathrm{St}(n,p)\cap\Sigma_{U,\widetilde{U}}$, is represented as well .}
    \label{fig:sphericalcap}
\end{figure}
\cref{thm:linearbound} provides a first upper bound on $M_\mathrm{E}$, namely $\widehat{M}_\mathrm{E}:[0,2\sqrt{p}]\rightarrow \mathbb{R}:\delta\mapsto\frac{\pi}{2}\delta$.
\subsection{A nonlinear upper bound on the Euclidean geodesic distance by the Frobenius distance for nearby matrices} 
As mentioned previously and numerically investigated in~\cref{fig:error_bound}, the relative error between the linear upper bound from~\cref{thm:linearbound} and $M_\mathrm{E}$ is large (up to $57\%$) for nearby matrices. We provide next the best upper bound $M_\mathrm{E}(\delta)$ for all $\delta\in[0,2]$ (where optimality is shown in~\cref{cor:reach_first_upper_bound}). 
\begin{thm}\label{thm:first_upper_bound}
\it
Let $n\geq2p$. For all $U,\widetilde{U}\in\mathrm{St}(n,p)$, if $\|\widetilde{U}-U\|_\mathrm{F}\leq 2$, then $d_\mathrm{E}(U,\widetilde{U})\leq 2\arcsin\left(\frac{\|\widetilde{U}-U\|_\mathrm{F}}{2}\right)$.
\end{thm}
\begin{proof}
We rely on the property from \cite[Cor.~3.4]{zimmermann2024injectivity}. The authors showed that for every Euclidean geodesic on $\mathrm{St}(n,p)$, there is $a_i,b_i\in\mathbb{R}$ with $i=1,...,p^2$ such that it can be embedded in the following normal form:
\begin{equation*}
	\gamma(t)=(a_1\cos(b_1 t), a_1\sin(b_1 t), ..., a_{p^2}\cos(b_{p^2} t), a_{p^2}\sin(b_{p^2} t) )^T \in\mathbb{R}^{2p^2}.
\end{equation*}
The transformation mapping curves from $(\mathrm{St}(n,p),\langle\cdot,\cdot\rangle_\mathrm{E})$ into $(\mathbb{R}^{2p^2},\langle\cdot,\cdot\rangle_\mathrm{F})$ is an isometric embedding, i.e., their length is conserved.
W.l.o.g., if we consider unit-speed geodesics, $\|\dot{\gamma}(t)\|_\mathrm{F}^2 =\sum_{i=1}^{p^2} a_i^2 b_i^2 =1$, the curvature of $\gamma(t)$ is bounded from above by~$1$ for all $t\in\mathbb{R}$, i.e., $\|\ddot{\gamma}(t)\|_\mathrm{F}=\sum_{i=1}^{p^2} a_i^2 b_i^4\leq 1$. This implies $|b_i|\leq 1$ \cite[Thm.~3.5]{zimmermann2024injectivity}. Since $\gamma$ has unit-speed, for all time intervals $[0,t^*]$, we have $l_\mathrm{E}(\gamma) = \int_{0}^{t^*} \|\dot{\gamma}(t)\|_\mathrm{F}\ \mathrm{d}t= t^*$. Let $U\simeq\gamma(0)$ and $\widetilde{U}\simeq\gamma(t^*)$ with $\|\widetilde{U}-U\|_\mathrm{F}\leq 2$ and let $\gamma$ be a minimal geodesic ($t^*=d_\mathrm{E}(U,\widetilde{U})$). It follows from the previous properties that 
\begin{align*}
	\|\widetilde{U}-U\|_\mathrm{F}^2 &= \sum_{i=1}^{p^2}a_i^2\left[(1-\cos(b_i t^*))^2+\sin(b_i t^*)^2\right]\\
										&= 4  \sum_{i=1}^{p^2} a_i^2 \sin^2\left(\frac{b_i t^*}{2}\right)\\
										&\geq 4\sin^2\left(\frac{t^*}{2}\right) \sum_{i=1}^{p^2} a_i^2 b_i^2\quad\quad\text{See below.}\\
										&=4 \sin^2\left(\frac{d_\mathrm{E}(U,\widetilde{U})}{2}\right)\quad\quad\ \text{ Unit-speed minimizing geodesic}.
\end{align*}
We leveraged that, if $ |b|\leq 1$, then $|\sin(b\frac{t}{2})|\geq |b|\sin(\frac{t}{2})$ for all $t\in[0,\pi]$ and $t^*\leq\pi $ follows by~\cref{thm:linearbound}. Finally, since $\|\widetilde{U}-U\|_\mathrm{F}\leq 2$, 
we have
\begin{equation}
\nonumber
	d_\mathrm{E}(U,\widetilde{U})\leq 2\arcsin\left(\frac{\|\widetilde{U}-U\|_\mathrm{F}}{2}\right).
\end{equation}
This concludes the proof.
\end{proof}
In view of~\cref{thm:linearbound} and~\cref{thm:first_upper_bound}, we define
\begin{equation}\label{eq:M_hat_E}
\widehat{M}_\mathrm{E}:[0,2\sqrt{p}]\rightarrow\mathbb{R}:\delta\mapsto\begin{cases}
2\arcsin\left(\frac{\delta}{2}\right) \text{ if } \delta \leq 2,\\
\frac{\pi}{2}\delta\quad \quad\quad \quad\text{ otherwise}.
\end{cases}
\end{equation}
\subsection{The Euclidean diameter of Stiefel} \label{sec:diameter} In~\cref{fig:shootingfig}, the lower and upper bounds cross at the upper right corner of the plot. This provides the diameter of the Stiefel manifold endowed with the Euclidean metric ($\beta=1$). This is shown next.

    \begin{thm}\label{thm:Euclideandiameter}
    \it
        Let $n\geq 2p$. The diameter of $(\mathrm{St}(n,p), \langle\cdot,\cdot\rangle_\mathrm{E})$ is $\pi\sqrt{p}$. Moreover, for all $U,\widetilde{U}\in \mathrm{St}(n,p)$, we have $d_{\mathrm{E}}(U,\widetilde{U})=\pi\sqrt{p}$ if and only if $\widetilde{U}=-U$.
    \end{thm}

\begin{proof} For all $U,\widetilde{U}\in \mathrm{St}(n,p)$, since $\|U-\widetilde{U}\|_\mathrm{F}\leq2\sqrt{p}$, we know by~\cref{thm:linearbound} that $d_{\mathrm{E}}(U,\widetilde{U})\leq\frac{\pi}{2}2\sqrt{p}=\pi\sqrt{p}$. Moreover,~\cref{thm:Euclideanlowerbound} yields the lower bound 
\begin{equation*}
    d_{\mathrm{E}}(U,-U)\geq2\sqrt{p}\arcsin\left(\frac{\|2U\|_{\mathrm{F}}}{2\sqrt{p}}\right)=\pi\sqrt{p}.
\end{equation*}
It follows that $d_{\mathrm{E}}(U,-U)=\pi\sqrt{p}$ and the diameter of $(\mathrm{St}(n,p),\langle\cdot,\cdot\rangle_\mathrm{E})$ is $\pi\sqrt{p}$. Moreover, for all $\widetilde{U}\neq -U$, $\|U-\widetilde{U}\|_\mathrm{F}<2\sqrt{p}$ and thus $d_{\mathrm{E}}(U,\widetilde{U})<\pi\sqrt{p}$. Therefore, $d_{\mathrm{E}}(U,\widetilde{U})=\pi\sqrt{p}$ if and only if $\widetilde{U}=-U$.
\end{proof}
We do not have at this time a proof for the diameter of the Stiefel manifold for other choices of $\beta$-metrics.
\subsection{Attaining the linear upper bound on the Euclidean geodesic distance by the Frobenius distance}\label{sec:reachingupperbound}

Looking back at~\cref{fig:shootingfig}, we just showed that the linear upper bound (dashed line) is attained at the two points where it meets the lower bound (solid line), i.e., $\widehat{M}_\mathrm{E}(\delta)=M_\mathrm{E}(\delta)$ at $\delta=0$ and $\delta=2\sqrt{p}$. However,~\cref{fig:shootingfig} suggests the existence of other points where the upper bound is attained, specifically at $\delta=2\sqrt{k}$ for $k=1,...,p-1$. We exhibit $2^p$ pairs ($U,\widetilde{U}$) achieving those points. 

Recall the definition~\eqref{eq:planeofcap} of the hyperplane $\Pi_{U,\widetilde{U}}$. $\Pi_{U,\widetilde{U}}$ divides $\mathbb{S}^{n\times p}_{\sqrt{p}}$ in two sets, as illustrated in~\cref{fig:sets}. We term \emph{open spherical cap} the set $\{X\in\mathbb{S}^{n\times p}_{\sqrt{p}}\ | \langle X, C\rangle_\mathrm{F} >\langle U, C\rangle_\mathrm{F}\}$.~\cref{lem:outsideball} shows that if a curve never enters the open spherical cap, its length can be bounded from below by the linear upper bound from~\cref{thm:linearbound}. 

    \begin{lem}\label{lem:outsideball}
        \it
        For all $U,\widetilde{U}\in \mathrm{St}(n,p)$ with $\widetilde{U}\neq -U$, all $\gamma:[0,1]\rightarrow\mathrm{St}(n,p)$ going from $U$ to $\widetilde{U}$, if $\langle \gamma(t), C\rangle_\mathrm{F}\leq \langle U, C\rangle_\mathrm{F}$ for all $t\in [0,1]$, then $l_\mathrm{E}(\gamma)\geq \frac{\pi}{2}\|\widetilde{U}-U\|_\mathrm{F}$.
    \end{lem}
\begin{proof}
    We recall the notation $\m{C} := \frac{U+\widetilde{U}}{2}$ and $r := \frac{\|U-\widetilde{U}\|_\mathrm{F}}{2}=\sqrt{\langle R, R \rangle_\mathrm{F}}$. Having $\langle \gamma(t), C\rangle_\mathrm{F}\leq \langle U, C\rangle_\mathrm{F}$ is equivalent to $
        \langle \gamma(t)-\m{C},\gamma(t)-\m{C}\rangle_\mathrm{F}\geq r^2$. Indeed, \begin{align*}
        r^2&\leq\langle \gamma(t)-\m{C},\gamma(t)-\m{C}\rangle_\mathrm{F}\\
        \iff 0 &\geq 2 \langle \gamma(t), C\rangle_\mathrm{F}- \langle \gamma(t), \gamma(t)\rangle_\mathrm{F} - \langle C, C \rangle_\mathrm{F}+\langle R, R\rangle_\mathrm{F}\\
        \iff 0 &\geq \langle \gamma(t), C\rangle_\mathrm{F} - \langle U, C\rangle_\mathrm{F}.
    \end{align*}
    Thus, $\gamma$ never enters the interior of the ball delimited by the sphere $\Sigma_{\widetilde{U},U}$ in $\mathbb{R}^{n\times p}$. 
Since $\gamma$ starts and ends on~$\Sigma_{\widetilde{U},U}$, it cannot be shorter than the shortest path on $\Sigma_{\widetilde{U},U}$. The latter has Euclidean length $\pi r = \frac{\pi}{2}\|\widetilde{U}-U\|_\mathrm{F} $. Therefore, we have $l_\mathrm{E}(\gamma)\geq \frac{\pi}{2}\|\widetilde{U}-U\|_\mathrm{F}$.
\end{proof}
The motivation behind~\cref{lem:outsideball} stems from our prior knowledge of  $d_\mathrm{E}(U,\widetilde{U})\leq \frac{\pi}{2}\|\widetilde{U}-U\|_\mathrm{F}$ (see~\cref{thm:linearbound}). Hence, if we find pairs of points $U, \widetilde{U}\in\mathrm{St}(n,p)$ such that the open spherical cap admits no point of $\mathrm{St}(n,p)$, namely \begin{equation}\label{eq:hole}
    \mathrm{St}(n,p)\cap\left\{\m{X}\in\mathbb{R}^{n\times p}\ | \ \langle X, C\rangle_\mathrm{F}> \langle U, C\rangle_\mathrm{F}\right\}  = \emptyset,
\end{equation}
we ensure~\cref{lem:outsideball} to hold for all curves $\gamma:[0,1]\rightarrow\mathrm{St}(n,p)$ going from $U$ to $\widetilde{U}$. In consequence, we will have $d_\mathrm{E}(U,\widetilde{U})= \frac{\pi}{2}\|\widetilde{U}-U\|_\mathrm{F}$ by leveraging~\cref{lem:outsideball} and~\cref{thm:linearbound}. Let us build such pairs. 
Take $k\in\{1,...,p\}$. For all $U\in \mathrm{St}(n,p)$, we define $h_k:\mathbb{R}^{n\times p}\rightarrow \mathbb{R}^{n\times p}$ by
\begin{equation}\label{eq:h_kdefinition}
    h_k(U) :=U \begin{bmatrix}
    -I_k&0\\
    0&I_{p-k}
    \end{bmatrix},
\end{equation} 
i.e., $h_k(U)$ ``flips'' the first $k$ columns of $U$. We show in~\cref{thm:reachlinearbound} that these pairs of matrices reach the upper bound.
    \begin{thm}\label{thm:reachlinearbound}
    \it
        Let $n\geq2p$. For all $U\in \mathrm{St}(n,p)$ and $k\in\{1,...,p\}$, we have $d_\mathrm{E}(U,h_k(U)) = \frac{\pi}{2}\|h_k(U)-U\|_\mathrm{F}=\pi\sqrt{k}$ and $\widehat{M}_\mathrm{E}(\delta) = M_\mathrm{E}(\delta)$ at $\delta=2\sqrt{k}$.
    \end{thm}
\begin{proof}
\cref{thm:reachlinearbound} is already proven for $k=p$ in~\cref{thm:Euclideandiameter}. Hence, we assume $k\leq p-1$ so that $h_k(U)\neq -U$. 
By definition of $h_k(U)$, we have
\begin{equation*}
    C:=\frac{1}{2}\left(U + h_k(U)\right) = U\begin{bmatrix}
    0_{k}&0\\
    0&I_{p-k}
    \end{bmatrix}.
\end{equation*} 
Hence, for all $\m{X}\in \mathrm{St}(n,p)$, we have $\langle X, C\rangle_\mathrm{F}\leq p-k$. 
Moreover, $\langle U, C \rangle_\mathrm{F} = p-k$ by definition. This yields $\langle X, C\rangle_\mathrm{F}\leq \langle U, C \rangle_\mathrm{F}$ for all $X\in \mathrm{St}(n,p)$.
Therefore, we obtain $d_\mathrm{E}(U,h_k(U)) = \frac{\pi}{2}\|h_k(U)-U\|_\mathrm{F}=\pi\sqrt{k}$ by leveraging~\cref{lem:outsideball} and~\cref{thm:linearbound}. 
\end{proof}
\begin{rem}
    In~\eqref{eq:h_kdefinition}, $h_k$ is arbitrarily defined as a function flipping the $k$ first columns of $U$. It is clear however that we can flip any subset of $k$ columns among $p$ and~\cref{thm:reachlinearbound} still holds. For each value of $k$, this makes $\left(\begin{smallmatrix}p\\ k\end{smallmatrix}\right):=\frac{p!}{k!(p-k)!}$ possible choices. Hence, for a given $U$, we highlighted a total of $2^p = \sum_{k=0}^p\left(\begin{smallmatrix}p\\ k\end{smallmatrix}\right)$ matrices reaching the bound.
\end{rem}

By~\cref{thm:reachlinearbound}, it is sufficient to find a geodesic of length $\pi\sqrt{k}$ between $U$ and $h_k(U)$ to ensure its minimality. To conclude this section, we build such geodesics. For all $U\in\mathrm{St}(n,p)$, all $\theta\in\mathbb{R}$, let $G_l(\theta)\in\mathrm{SO}(l)$ from \eqref{eq:G_theta} and $u_\perp\in\mathrm{St}(n,1)$ with $u_\perp^TU=0$, and let 
\begin{equation*}
        \breve{A} := \frac{\pi\sqrt{3}}{3}\begin{bmatrix}
            0&1&-1\\
            -1&0&1\\
            1&-1&0
        \end{bmatrix}\text{ and } \breve{B}:=\frac{\pi\sqrt{3}}{3}\begin{bmatrix}
            1&1&1
        \end{bmatrix}.
    \end{equation*} 
Finally, for $k=1,...,p$, we define the curves
\begin{align}\label{eq:Euclidean_to_hk}
\gamma_k:[0,1]&\rightarrow\mathrm{St}(n,p)\\
		\nonumber
        t&\mapsto\begin{cases}
        [U_{1:k}\m{G}_k(t\pi)\ \ U_{k+1:p}]\hspace{1.8cm}\text{ if $k$ is even,}\\
        [\cos(t\pi)U_{1}+\sin(t\pi)U_{\perp} \quad U_{2:p}]\quad \text{ if } k=1, \text{ otherwise,}\\
         [U_{1:k-3}\m{G}_{k-3}(t\pi)\ \ [
            U_{k-2:k}\ U_{\perp}
        ]\exp\left( t\left[\begin{smallmatrix}
            2\breve{A}&-\breve{B}^T\\
            \breve{B}&0
        \end{smallmatrix}\right]\right)\m{I}_{4\times 3}\exp(-\breve{A}t)\ \ U_{k+1:p}].
        \end{cases}
\end{align}
\begin{thm}\label{thm:Euclidean_minimality}
Let $n\geq 2p$. For all $U\in\mathrm{St}(n,p)$ and all $k\in\{1,...,p\}$, it holds that $\gamma_k$ from \eqref{eq:Euclidean_to_hk} is a minimal geodesic from $U$ to $h_k(U)$. 
\end{thm}
\begin{proof}
In view of~\cref{thm:reachlinearbound}, it is enough to show that $\gamma_k(0)=U$, $\gamma_k(1)=h_k(U)$ and $l_\mathrm{E}(\gamma_k)=\pi\sqrt{k}$.

If $k$ is even, let $A_k(\pi)$ be defined as in \eqref{eq:Atheta}. Then, $l_\mathrm{E}(\gamma_k) = \|A_k(\pi)\|_\mathrm{E} = \|A_k(\pi)\|_\mathrm{F} = \pi\sqrt{k}$. It is clear that  $\gamma_k(0)=U$ and $\gamma_k(1)=h_k(U)$. Hence, $\gamma_k$ is a minimal geodesic.

If $k=1$, it holds that $\gamma_1(0)=U$, $\gamma_1(1)=h_1(U)$. Moreover, $l_\mathrm{E}(\gamma_1) = \pi$.

If $k\geq 3$ and $k$ odd, again, $\gamma_k(0)=U$. It is less obvious that $\gamma_k(1)=h_k(U)$ and we prove it in~\cref{app:Euclideanexponential} since it follows exclusively from computations.  In addition, $l_\mathrm{E}(\gamma_k) = \sqrt{\|A_{k-3}(\pi)\|_\mathrm{F}^2+\|\breve{A}\|_\mathrm{F}^2+\|\breve{B}\|_\mathrm{F}^2} = \pi\sqrt{(k-3) + 2 + 1} = \pi\sqrt{k}$.
\end{proof}
\begin{thm}\label{thm:Euclidean_geodesics_dist}
    Let $n\geq2p$. For all $U\in\mathrm{St}(n,p)$, all $k\in\{1,...,p\}$, if $\gamma_k$ is the geodesic from \eqref{eq:Euclidean_to_hk}, then $d_\mathrm{E}(U,\gamma_k(t))= 2\sqrt{k}\arcsin\left(\frac{\|U-\gamma_k(t)\|_\mathrm{F}}{2\sqrt{k}}\right)$. 
\end{thm}
\begin{proof}
We know from~\cref{thm:Euclidean_minimality} that $\gamma_k$ is minimal. Hence, $d_\mathrm{E}(U,\gamma_k(t)) =t\|\dot{\gamma}_k(0)\|_{\mathrm{E}}= t \pi	\sqrt{k}$.
We divide the problem in subcases as previously.

If $k$ is even, replacing $p$ by $k$ in~\cref{thm:lowerboundreached} and its proof is enough to conclude.

If $k=1$, we have 
$$
\|U-\gamma_1(t)\|_\mathrm{F}^2=\|(1-\cos(t\pi)U_1-\sin(t\pi)u_\perp\|^2_\mathrm{F}=4\sin^2\left(\frac{t\pi}{2}\right).$$
This leads to $t\pi=d_\mathrm{E}(U,\gamma_1(t))= 2\arcsin\left(\frac{\|U-\gamma_1(t)\|_\mathrm{F}}{2}\right)$.

If  $k\geq 3$ and odd, let $M:=\left[\begin{smallmatrix}
            2\breve{A}&-\breve{B}^T\\
            \breve{B}&0
        \end{smallmatrix}\right]$. We have
\begin{align*}\|U-\gamma_k(t)\|_\mathrm{F}^2 &= \|I_{k-3}-G(t\pi)\|^2_\mathrm{F}+\|I_{4\times 3}-\exp\left( tM\right)\m{I}_{4\times 3}\exp(-\breve{A}t)\|^2_\mathrm{F}\\
        &=2(k-3)(1-\cos(t\pi))+2\left[3-\Tr\left(I_{3\times 4}\exp\left( tM\right)\m{I}_{4\times 3}\exp(-\breve{A}t)\right)\right]\\
        &=2k(1-\cos(t\pi))=4k\sin^2\left(\frac{d_\mathrm{E}(U,\gamma_k(t))}{2\sqrt{k}}\right).
\end{align*}
The fact that $\Tr\left(I_{3\times 4}\exp\left( tM\right)\m{I}_{4\times 3}\exp(-\breve{A}t)\right)=3\cos(t\pi)$ can be obtained analytically either by using the eigenvalue decomposition of $tM$ and $t\breve{A}$ from~\cref{app:Euclideanexponential}, or by using the Laplace transform. In both cases, this leads to quite tedious but feasible computations. Hence, $d_\mathrm{E}(U,\gamma_k(t))= 2\sqrt{k}\arcsin\left(\frac{\|U-\gamma_k(t)\|_\mathrm{F}}{2\sqrt{k}}\right)$. This concludes the proof.
\end{proof}

The evolution of the geodesic distance in terms of the Frobenius distance for the minimal geodesics $\gamma_k$ is illustrated by solid black lines in~\cref{fig:tominusU} for $k=1, 2, 3, 4$ in $\mathrm{St}(5,4)$. We transgress willingly the condition $n\geq2p$ to observe numerically that the upper bound is satisfied nonetheless. 
\begin{figure}
    \centering
    \includegraphics[width = 8cm]{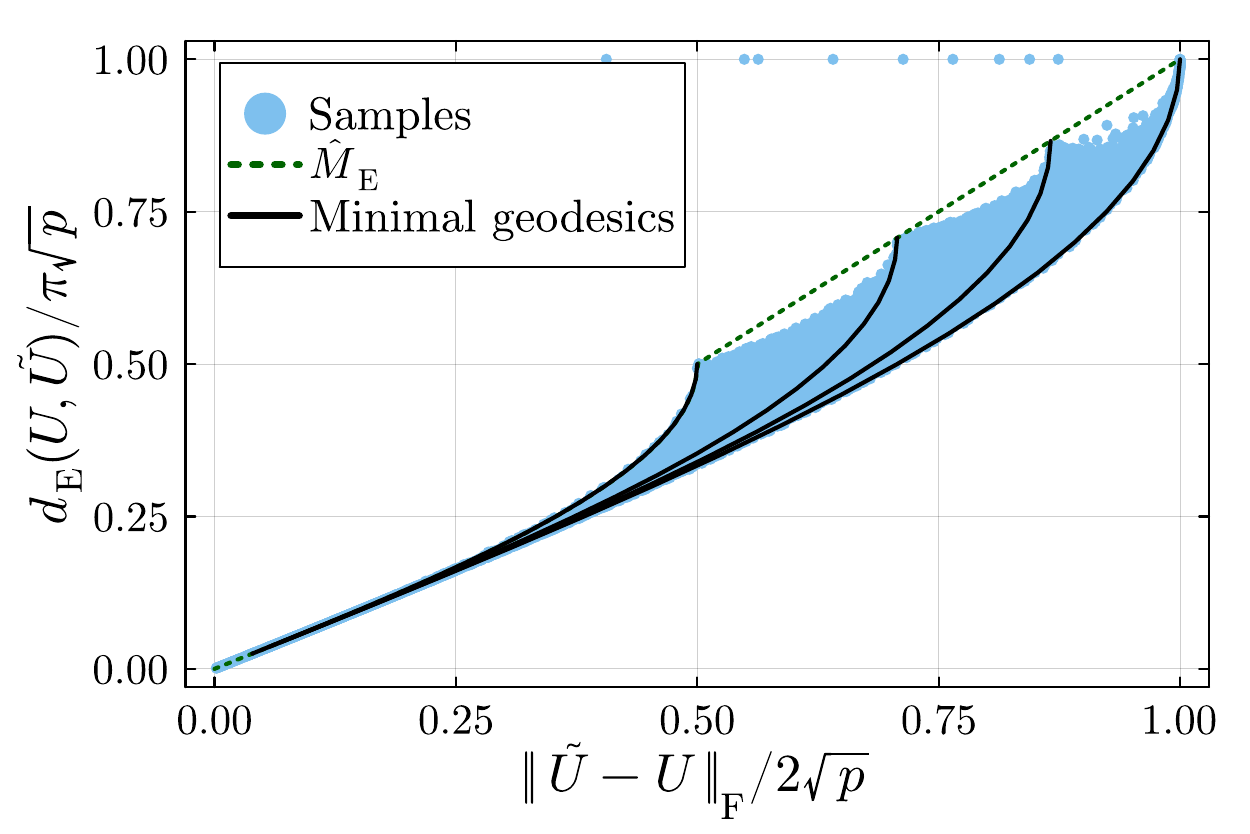}
    \caption{Evolution of the Euclidean geodesic distance ($\beta=1$) in terms of the Frobenius distance for  minimal geodesics going  from $U$ to $h_k(U)$ for $k=1, 2, 3, 4$ in $\mathrm{St}(5,4)$ (solid black lines). The upper bound $\widehat{M}_\mathrm{E}$ is represented by the green dashed line. Notice that the condition $n\geq 2p$ is not satisfied but~\cref{thm:linearbound} and~\cref{thm:first_upper_bound} are verified nonetheless, except for failures of the method. If $\delta:=\|U-\widetilde{U}\|_\mathrm{F}$,~\cref{thm:Euclidean_geodesics_dist} shows that the solid black curves have the shape $[0,2\sqrt{k}]\ni\delta\mapsto2\sqrt{k}\arcsin\left(\frac{\delta}{2\sqrt{k}}\right)$ in the $(\delta,d_\mathrm{E}(U,\widetilde{U}))$-plane.
    }
    \label{fig:tominusU}
\end{figure}

Previously,~\cref{thm:lowerboundreached} and \cref{thm:lowerboundreachedEuclideanetal} constrained either $p$ to be even or $n\geq2p$ to have $m_\beta=\widehat{m}_\beta$. In the particular case of the Euclidean metric ($\beta=1$),~\cref{thm:Euclidean_geodesics_dist} for $k=p$ shows that $m_\mathrm{E} = \widehat{m}_\mathrm{E}$ for all pairs $(n,p)$. This is explained in~\cref{cor:reachedforallp}.
\begin{cor}\label{cor:reachedforallp}
For all $U\in\mathrm{St}(n,p)$, all $\delta\in[0,2\sqrt{p}]$, there is $\widetilde{U}\in\mathrm{St}(n,p)$ with $\|U-\widetilde{U}\|_\mathrm{F}=\delta$ such that the lower bound from~\cref{thm:Euclideanlowerbound} is attained, i.e.,  $d_\mathrm{E}(U,\widetilde{U})= 2\sqrt{p}\arcsin\left(\frac{\|U-\widetilde{U}\|_\mathrm{F}}{2\sqrt{p}}\right)$. Hence, $\widehat{m}_\mathrm{E} = m_\mathrm{E}$.
\end{cor}
\begin{proof}
Take $\widetilde{U}=\gamma_p(t^*)$ where $t^*\in[0,1]$ is such that $\|U-\widetilde{U}\|_\mathrm{F}=\delta$. Since $\| U-\gamma_p(0)\|_\mathrm{F}=0$ and $\| U-\gamma_p(1)\|_\mathrm{F}=2\sqrt{p}$, the intermediate value theorem ensures the existence of $t^*\in[0,1]$.~\cref{cor:reachedforallp} holds by~\cref{thm:Euclidean_geodesics_dist}. The condition $n\geq2p$ from~\cref{thm:Euclidean_geodesics_dist} can be dropped because~\cref{thm:Euclidean_minimality} is not needed to ensure the minimality of $\gamma_p$. Indeed, the minimality of $\gamma_p$ holds because $l(\gamma_p) = \pi\sqrt{p}=d_\mathrm{E}(U,-U)$ by \cref{thm:Euclideanlowerbound}.
\end{proof}

Moreover \cref{thm:Euclidean_geodesics_dist} for $k=1$ shows that  $\widehat{M}_\mathrm{E}(\delta) = M_\mathrm{E}(\delta)$ for all $\delta\in[0,2]$.
\begin{cor}\label{cor:reach_first_upper_bound}
Let $n\geq 2p$. For all $U\in\mathrm{St}(n,p)$, all $\delta\in[0,2]$, there is $\widetilde{U}\in\mathrm{St}(n,p)$ with $\|U-\widetilde{U}\|_\mathrm{F}=\delta$ such that the upper bound from~\cref{thm:first_upper_bound} is attained, i.e.,  $d_\mathrm{E}(U,\widetilde{U})= 2\arcsin\left(\frac{\|U-\widetilde{U}\|_\mathrm{F}}{2}\right)$. Said otherwise, $\widehat{M}_\mathrm{E}(\delta) = M_\mathrm{E}(\delta)$ for all $\delta\in[0,2]$.
\end{cor}
\begin{proof}
Take $\widetilde{U}=\gamma_1(t^*)$ where $t^*\in[0,1]$ is such that $\|U-\widetilde{U}\|_\mathrm{F}=\delta$. Since $\| U-\gamma_1(0)\|_\mathrm{F}=0$ and $\| U-\gamma_1(1)\|_\mathrm{F}=2$, the intermediate value theorem ensures the existence of $t^*\in[0,1]$.~\cref{cor:reach_first_upper_bound} holds by~\cref{thm:Euclidean_geodesics_dist}. 
\end{proof}
\section{Bounds on the $\beta$-distance by the Frobenius distance: summary and small distances}\label{sec:generalization_to_beta}
In view of the bilipschitz equivalence of the $\beta$-distances (\cref{cor:allbounds}), we define 
\begin{equation}\label{eq:M_hat_beta}
     \widehat{M}_\beta(\delta):=\max\{1,\beta\}\widehat{M}_\mathrm{E}(\delta)=\max\{1,\beta\}\begin{cases}
2\arcsin\left(\frac{\delta}{2}\right) \text{ if } \delta \leq 2,\\
\frac{\pi}{2}\delta\quad \quad\quad \quad\text{ otherwise}.
\end{cases}
\end{equation} 
$\widehat{M}_\mathrm{E}$ was defined in~\eqref{eq:M_hat_E}.~\cref{thm:generalizedbounds} generalizes the bounds from~\cref{thm:Euclideanlowerbound},~\cref{thm:linearbound} and~\cref{thm:first_upper_bound} to all $\beta$-distances by leveraging the bilipschitz equivalence from~\cref{cor:allbounds}. The bounds are illustrated in~\cref{fig:shootingfig}.
\begin{thm}\label{thm:generalizedbounds}
	\it
	 Let $n\geq 2p$. Then for all $\beta>0$, all $U,\widetilde{U}\in \mathrm{St}(n,p)$, we have $$\widehat{m}_\beta(\|\widetilde{U}-U\|_{\mathrm{F}})\leq d_{\beta}(\widetilde{U},U)\leq \widehat{M}_\beta(\|\widetilde{U}-U\|_{\mathrm{F}}),$$
	 where $\widehat{m}_\beta$ and $\widehat{M}_\beta$ are defined in~\eqref{eq:m_hat_beta} and~\eqref{eq:M_hat_beta} respectively.
\end{thm}
\begin{proof}
	The lower bound is a reminder of \cref{cor:betalowerbound}. The upper bound is obtained by combining~\cref{thm:linearbound} and~\cref{thm:first_upper_bound} and~\cref{cor:allbounds}.
\end{proof}
By relaxing~\cref{thm:generalizedbounds}, we can highlight that all the $\beta$-distances are bilipschitz equivalent to the Frobenius distance. This result is presented for the sake of completeness although there is no reason to prefer~\cref{cor:Frob_equi_geo} to~\cref{thm:generalizedbounds} in practice.
\begin{cor}\label{cor:Frob_equi_geo}
        \it
         Let $n\geq 2p$. Then for all $\beta>0$, all $U,\widetilde{U}\in \mathrm{St}(n,p)$, we have
         $$\min\{1,\sqrt{\beta}\}\|\widetilde{U}-U\|_{\mathrm{F}}\leq d_\beta(\widetilde{U},U)\leq \max\{1,\sqrt{\beta}\}\frac{\pi}{2}\|\widetilde{U}-U\|_{\mathrm{F}}.$$
\end{cor}

\begin{proof}
    By~\cref{thm:generalizedbounds}, we have
    \begin{equation*}
    d_{\mathrm{E}}(\widetilde{U},U)\geq\min\{1,\sqrt{\beta}\} 2\sqrt{p}\arcsin\left(\frac{\|\widetilde{U}-U\|_{\mathrm{F}}}{2\sqrt{p}}\right)\geq\min\{1,\sqrt{\beta}\} \|\widetilde{U}-U\|_{\mathrm{F}}.
    \end{equation*}
     The last inequality holds since $[0,2\sqrt{p}]\ni \delta\mapsto \min\{1,\sqrt{\beta}\}2\sqrt{p}\arcsin\left(\frac{\delta}{2\sqrt{p}}\right)$ is a convex function and $\delta\mapsto \min\{1,\sqrt{\beta}\}\delta$ is the tangent at $\delta=0$.
\end{proof}

Notice that the value of $\beta$ such that the two Lipschitz constants of~\cref{cor:Frob_equi_geo} are the closest from each other is $\beta=1$, i.e., under the Euclidean metric. As one could have expected, the Euclidean geodesic distance is the most similar to the Frobenius distance.

\textbf{Growth of the $\beta$-distance for nearby matrices.} We complete our picture of the $\beta$-distance by characterizing its initial growth factor in terms of the Frobenius distance along any $\beta$-geodesic. In~\cref{fig:shootingfig},~\cref{fig:canonicalsamples} and \cref{fig:tominusU}, this translates by the minimum and maximum slopes of the dotted domain at the origin.
For all $U\in\mathrm{St}(n,p)$, all $\beta>0$, let $[0,\epsilon_\beta]$ with $\epsilon_\beta>0$ be a time interval such that all unit-speed $\beta$-geodesics emanating from $U$ monotonically increase the Frobenius distance. The existence of $\epsilon_\beta>0$ is guaranteed since, otherwise,~\cref{cor:Frob_equi_geo} would imply the existence of a unit-speed geodesic loop of length zero. By \cite[Eq.~2.1]{absil2024ultimate}, the injectivity radius $\mathrm{inj}_{\mathrm{St}(n,p),\beta}$ would be zero. Since the injectivity radius is known to be positive \cite[III.4,~Prop.~4.13]{sakai1996riemannian}, $\epsilon_\beta>0$ exists. 
For all $\Delta\in T_U\mathrm{St}(n,p)$ with $\|\Delta\|_\beta=1$, let $\delta_{\beta}(\Delta) :=\|U-\mathrm{Exp}_{\beta,U}(\epsilon_\beta\Delta)\|_\mathrm{F}$. Then, $[0,\epsilon_\beta]\ni t\mapsto\mathrm{Exp}_{\beta,U}(t\Delta)$ can be time-reparameterized by the monotonically increasing function $s(t):=\|U-\mathrm{Exp}_{\beta,U}(t\Delta)\|_\mathrm{F}$ such that $[0,\delta_\beta(\Delta)]\ni \delta\mapsto\mathrm{Exp}_{\beta,U}(s^{-1}(\delta)\Delta)=:\widetilde{U}_\Delta(\delta)$ satisfies $\delta=\|U-\widetilde{U}_\Delta(\delta)\|_\mathrm{F}$. Finally, for all $\Delta\in T_U\mathrm{St}(n,p)$ with $\|\Delta\|_\beta=1$ and $\delta\in[0,\delta_\beta(\Delta)]$, we have
\begin{align}
\nonumber
\widehat{m}_\beta(\delta)&\leq d_{\beta}(\widetilde{U}_\Delta(\delta),U)\leq \widehat{M}_\beta(\delta)\\
\nonumber
\Longrightarrow\lim_{\delta\rightarrow0} \frac{\widehat{m}_\beta(\delta)}{\delta}&\leq \lim_{\delta\rightarrow0} \frac{d_{\beta}(\widetilde{U}_\Delta(\delta),U)}{\delta} \leq \lim_{\delta\rightarrow0} \frac{\widehat{M}_\beta(\delta)}{\delta}\\
\label{eq:slope_beta}
  \Longrightarrow\min\{1,\sqrt{\beta}\}&\leq\lim_{\delta\rightarrow0} \frac{d_{\beta}(\widetilde{U}_\Delta(\delta),U)}{\delta}\leq \max\{1,\sqrt{\beta} \}.
\end{align}
For all $U\in \mathrm{St}(n,p)$, if $\beta\leq 1$, the lower bound is attained for $\Delta=UA$ and the upper bound for $\Delta=U_\perp B$. If $\beta\geq 1$, then the lower bound is attained for $\Delta=U_\perp B$ and the upper bound for $\Delta = UA$. The bounds from~\eqref{eq:slope_beta} are numerically corroborated in~\cref{fig:range}.

\begin{figure}
\centering
\includegraphics[width = 7cm]{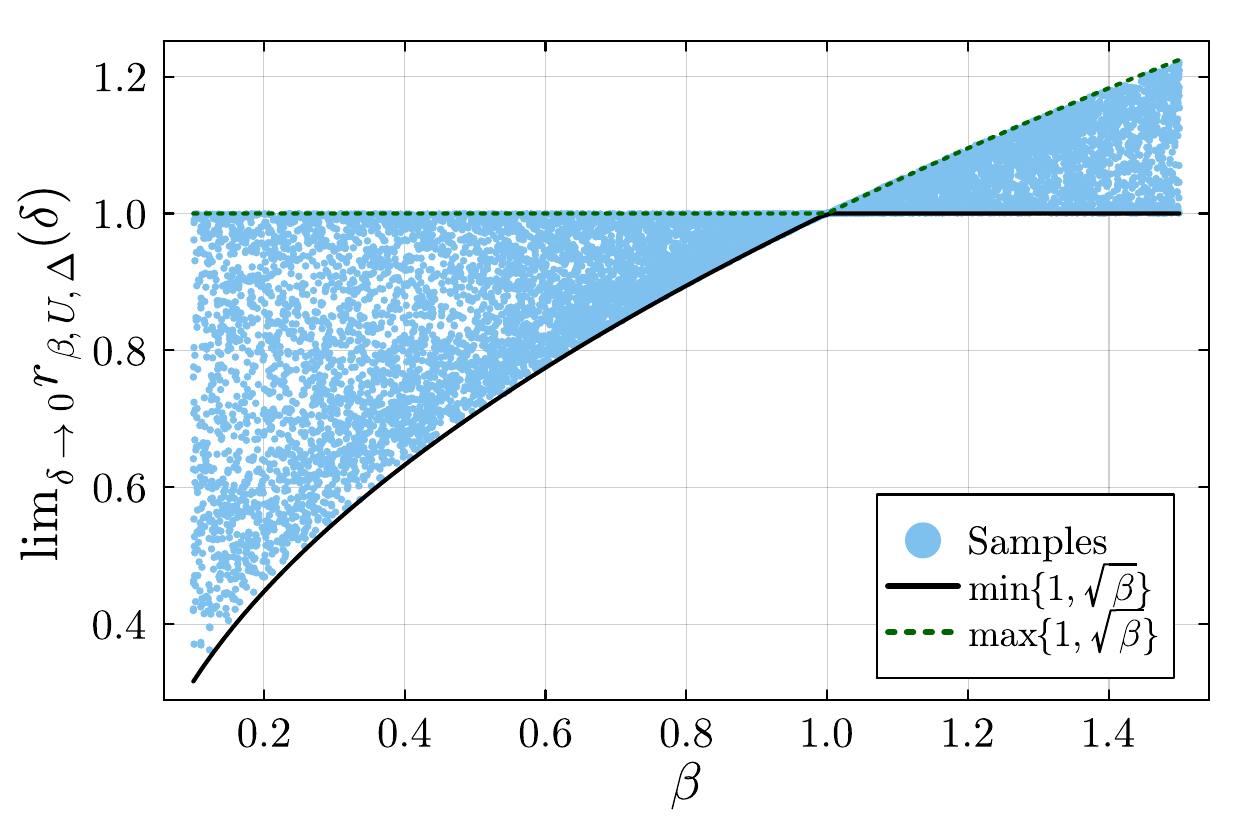}
\caption{Numerical experiment on $\mathrm{St}(4, 2)$ where we compute $\lim_{\delta\rightarrow 0}r_{\beta,U,\Delta}(\delta):=\lim_{\delta\rightarrow 0}\frac{d_\beta(U, U_{\Delta}(\delta)}{\delta}$ for 10000 randomly chosen $U\in\mathrm{St}(4, 2)$, $\Delta\in T_U\mathrm{St}(4, 2)$ and $\beta\in[0.1,1.5]$. The limit $\delta\rightarrow 0$ is evaluated at $\delta=10^{-6}$. This experiment corroborates \eqref{eq:slope_beta}: $\min\{1,\sqrt(\beta)\}\leq\lim_{t\rightarrow 0}r_{\beta,U,\Delta}(\delta)\leq \max\{1,\sqrt(\beta)\}$.}
\label{fig:range}
\end{figure}
\section{Conclusion} We have presented new bounds on the geodesic distances on the Stiefel manifold endowed with any member of the family of Riemannian metrics introduced by H\"uper et al.~\cite{Huper2021}. We started by quantifying the influence of the parameter $\beta$ on the geodesic distance in~\cref{thm:equivalence}. This result enlightens the choice of the metric and its influence on numerical results. New bounds involving the Frobenius distance help improving the performance of minimal geodesic computation algorithms. For any two points $U,\widetilde{U}\in\mathrm{St}(n,p)$, the output should satisfy \cref{thm:generalizedbounds}. Moreover, \cref{thm:generalizedbounds} quantifies the error made by the Frobenius distance as an approximation of the geodesic distance. For example, it follows that $\|U-\widetilde{U}\|_\mathrm{F}\leq 1$ yields $1\leq \frac{d_\mathrm{E}(U,\widetilde{U})}{\|U-\widetilde{U}\|_\mathrm{F}}\leq 1.05$. Due to \cref{lem:augmentedrows}, the constraint $n\geq 2p$ appears recurrently in this work. Examples, such as \cref{app:branches}, indicate that $n\geq 2p$ can not be relaxed when $\beta>1$. However, \cref{lem:augmentedrows} is only leveraged with $\beta=1$  for which \cref{conj:rows_beta}, if true, allows to drop this constraint throughout the article. Finally, we pointed out the relevance of the sets $\Sigma_{U,\widetilde{U}}$ and $\Pi_{U,\widetilde{U}}$ in the analysis of the Stiefel manifold since \cref{thm:reachlinearbound} revealed non-trivial open regions of the ambient space $\mathbb{R}^{n\times p}$ where no points from the Stiefel manifold exist. In particular, these open regions explain the maximum sectional curvature $1$ of the Stiefel manifold under the Euclidean metric~\cite[Thm.~10]{zimmermann2024high}.
\appendix
\section{Riemannian exponential computation}\label{app:Euclideanexponential}
\begin{lem}
For all $U\in\mathrm{St}(n,3)$, $u_\perp\in\mathrm{St}(n,1)$ with $u_\perp^TU=0$, if $\Delta=UA+u_\perp B$ with 
\begin{equation*}
\m{A} := \frac{\pi\sqrt{3}}{3}\begin{bmatrix}
            0&1&-1\\
            -1&0&1\\
            1&-1&0
        \end{bmatrix}\text{ and } \m{B}:=\frac{\pi\sqrt{3}}{3}\begin{bmatrix}
            1&1&1
        \end{bmatrix},
    \end{equation*}
    then $\mathrm{Exp}_{\mathrm{E},U}(\Delta) = -U$.
    \end{lem}
    \begin{proof}
    By \eqref{eq:paramgeogeneral}, we have
    \begin{equation}\label{eq:exp1}
    \mathrm{Exp}_{\mathrm{E},U}(\Delta) = [U\ u_\perp]\exp\left(\begin{bmatrix}2A&-B^T\\
    B&0\end{bmatrix}\right)I_{4\times 3}\exp(-A).
    \end{equation}
    The proof simply consists of evaluating analytically~\eqref{eq:exp1} using eigenvalue decompositions (EVD).
    We first compute an EVD of $A$. $\det(A-\lambda I_3)=-\lambda(\lambda^2+\pi^2)$, thus the eigenvalues are $0$, $i\pi$ and $-i\pi$. The unit-norm eigenvectors are given by
    \begin{equation*}
    	     v_0=\frac{\sqrt{3}}{3}\begin{bmatrix}1\\1\\1\end{bmatrix},\ v_{i\pi} = \frac{\sqrt{3}}{3}\begin{bmatrix}1\\-2+i\frac{\sqrt{3}}{2}\\-2-i\frac{\sqrt{3}}{2}\end{bmatrix},\ v_{-i\pi} = \frac{\sqrt{3}}{3}\begin{bmatrix}1\\-2-i\frac{\sqrt{3}}{2}\\-2+i\frac{\sqrt{3}}{2}\end{bmatrix}.
    	     \end{equation*}
    Letting $V_A$ the matrix of eigenvectors and $\Lambda_A$ the matrix of eigenvalues, we can compute
    \begin{equation}\label{eq:exp2}
    \exp(-A)=V_A\exp(-\Lambda_A)\overline{V_A}^T = \frac{1}{3}\begin{bmatrix}
    -1&2&2\\
    2&-1&2\\
    2&2&-1\\
    \end{bmatrix}.
    \end{equation} 
	Then defining $M=\begin{bmatrix}2A&-B^T\\
    B&0\end{bmatrix}$, we can compute $\det(M-\lambda I) = (\lambda^2+\pi^2)(\lambda^2+(2\pi)^2)$. The eigenvalues are $-i\pi, i\pi, -2i\pi$ and  $i2\pi$. The unit-norm eigenvectors are given by
    \begin{equation*}
       v_{-i\pi}=\begin{bmatrix}
       -i\frac{\sqrt{6}}{6}\\
       -i\frac{\sqrt{6}}{6}\\
       -i\frac{\sqrt{6}}{6}\\
       \frac{\sqrt{2}}{2}\\
       \end{bmatrix},\ v_{i\pi}=\begin{bmatrix}
       i\frac{\sqrt{6}}{6}\\
       i\frac{\sqrt{6}}{6}\\
       i\frac{\sqrt{6}}{6}\\
       \frac{\sqrt{2}}{2}
       \end{bmatrix},\ v_{-i2\pi}=\begin{bmatrix}
       -\frac{\sqrt{3}}{3}\\
       \frac{\sqrt{3}}{6}+i\frac{1}{2}\\
       \frac{\sqrt{3}}{6}-i\frac{1}{2}\\
       0
       \end{bmatrix},\ v_{i2\pi}=\begin{bmatrix}
       -\frac{\sqrt{3}}{3}\\
       \frac{\sqrt{3}}{6}-i\frac{1}{2}\\
       \frac{\sqrt{3}}{6}+i\frac{1}{2}\\
       0
       \end{bmatrix}.
    \end{equation*}
    Letting $V_M$ the matrix of eigenvectors and $\Lambda_M$ the matrix of eigenvalues, we can compute
    \begin{equation}\label{eq:exp3}
    \exp(M) = V_M \exp(\Lambda_M)\overline{V}_M^T= \frac{1}{3}
    \begin{bmatrix}
    1&-2&-2&0\\
    -2&1&-2&0\\
    -2&-2&1&0\\
    0&0&0&-3\\
    \end{bmatrix}.
    \end{equation}
    Combining \eqref{eq:exp1},~\eqref{eq:exp2} and \eqref{eq:exp3}, it is direct to conclude that $\mathrm{Exp}_{\mathrm{E},U}(\Delta)=-U$.
    \end{proof}
    \section{The two branches of~\cref{fig:canonicalsamples}}\label{app:branches}
    The two ``branches'' on the right plot of~\cref{fig:canonicalsamples} are a certificate of failure of \cite[Alg.~2]{ZimmermannRalf22} to find the minimal geodesic in $\mathrm{St}(3,2)$ between $U$ and $\widetilde{U}$ with $\widetilde{U}$ close to $-U$ for $\beta=2$. Indeed, $\widetilde{U}=-U$ is the unique point achieving $\|U-\widetilde{U}\|_\mathrm{F}=2\sqrt{p}$. Hence, the pair $U,-U$ realizes the tip of one of the two branches. From this pair, infinitesimal changes of the Frobenius distance would bring to the other branch and yield discontinuous changes of the geodesic distance. This contradicts the bilipschitz equivalence of the geodesic and the Frobenius distances. The upper branch is drawn, notably, by the geodesic curve $\gamma_1(t) =UG_{p}(t\pi)$ with $G_p$ defined as in \eqref{eq:G_theta}. For $t\in[0,1]$, this curve has length $l_\beta(\gamma_1) = \pi\sqrt{2\beta}$ exactly. However, one can find another curve of shorter length if $\beta>1$. Consider the curve
    \begin{equation}
    \gamma_2(t) = [U\ U_\perp]\exp\left(t\begin{bmatrix}
    2\beta A&-B^T\\
    B&0
    \end{bmatrix}\right)I_{3\times 2}\exp((1-2\beta)At),
    \end{equation}
    with $A = \begin{bmatrix}0&\frac{\pi}{1-2\beta}\\
    -\frac{\pi}{1-2\beta}&0
    \end{bmatrix}$ and $B\in\mathbb{R}^{1\times 2}$ such that $\|B\|_\mathrm{F} = 2\pi \sqrt{1-\frac{\beta^2}{(1-2\beta)^2}}$. It follows that 
$\left\|\left[\begin{smallmatrix}
    2\beta A&-B^T\\
    B&0
    \end{smallmatrix}\right]\right\|_\mathrm{F} = 2\pi\sqrt{2}$. Hence, the eigenvalues are $0,2i\pi,-2i\pi$ and $\exp\left[\begin{smallmatrix}
    2\beta A&-B^T\\
    B&0
    \end{smallmatrix}\right] =I_3$.
   Therefore, $\gamma_2(1) = [U\ U_\perp]I_3 I_{3\times 2}(-I_2) = -U$. Moreover, $\gamma_2$ has shorter length that $\gamma_1$:
   $$l_\beta(\gamma_2) =\sqrt{\beta \|A\|_\mathrm{F}^2+\|B\|_\mathrm{F}^2} = \pi\sqrt{\frac{2\beta}{(1-2\beta)^2}+4\left(1-\frac{\beta^2}{(1-2\beta)^2}\right)} < \pi\sqrt{2\beta}\text{ if } \beta>1.$$
  These two curves $\gamma_1$ and $\gamma_2$ from $U$ to $-U$, and associated families of neighboring curves joining $U$ to neighbors of $-U$ are the reason for the existence of the two branches on the right plot of~\cref{fig:canonicalsamples}. The upper branch is the result of convergence of \cite[Alg.~2]{ZimmermannRalf22} to the non-minimal family of geodesics.
  
Finally, this discussion indicates that \cref{lem:augmentedrows} does not hold for $n<2p$ when $\beta>1$. Indeed, let $U\in\mathrm{St}(2,2)=\mathrm{O}(2)$. The shortest geodesics on $\mathrm{O}(2)$ are known \cite[Eq.~2.14]{EdelmanArias98} and $\gamma_1$ is the shortest geodesic from $U$ to $-U$. Now, append a row of zeros to $U$ so that the new logarithm problem is defined on $\mathrm{St}(3,2)$. We have seen above that the newly available path $\gamma_2$ between $U$ and $-U$ is strictly shorter that $\gamma_1$. Hence, appending a row of zeros decreased $d_\beta(U,-U)$.

\bibliographystyle{splncs04}

\end{document}